\documentclass[11pt, reqno]{amsart}
\usepackage{amsmath, amsthm, amsfonts, amsbsy, amssymb, upref, enumerate, bigstrut}

\usepackage{epsfig, graphicx}
\usepackage{pdfsync, hyperref}

\newcommand{\I}{\mathrm{i}}

\newcommand{\ep}{\epsilon}

\newcommand{\mm}{\mathcal{M}}

\newcommand{\mc}{\mathcal{C}}

\newcommand{\mw}{\mathcal{W}}

\newtheorem{thm}{Theorem}[section]
\newtheorem{lmm}[thm]{Lemma}

\newcommand{\cc}{\mathbb{C}}

\newcommand{\dd}{\mathcal{D}}

\newcommand{\ee}{\mathbb{E}}

\newcommand{\ma}{\mathcal{A}}

\newcommand{\pp}{\mathbb{P}}

\newcommand{\ra}{\rightarrow}
\newcommand{\rr}{\mathbb{R}}

\newcommand{\var}{\mathrm{Var}}

\newcommand{\zz}{\mathbb{Z}}
\newcommand{\N}{\mathbb{N}} 
 
\newcommand{\fpar}[2]{\frac{\partial #1}{\partial #2}}
\newcommand{\spar}[2]{\frac{\partial^2 #1}{\partial #2^2}}
\newcommand{\mpar}[3]{\frac{\partial^2 #1}{\partial #2 \partial #3}}

\newcommand{\hx}{\hat{x}}
\newcommand{\HX}{\hat{X}}
\numberwithin{equation}{section}

\begin{document}
\title{Nonlinear large deviations}
\author{Sourav Chatterjee}
\address{\newline Department of Statistics \newline Stanford University\newline Sequoia Hall, 390 Serra Mall \newline Stanford, CA 94305\newline \newline \textup{\tt souravc@stanford.edu}\newline \textup{\tt adembo@stanford.edu}}
\thanks{Sourav Chatterjee's 
research partially supported by NSF grant DMS-1441513}
\author{Amir Dembo}
\thanks{Amir Dembo's 
research partially supported by NSF grant DMS-1106627}
\keywords{Large deviations, sparse random graphs, regularity lemma, arithmetic progressions, exponential random graph models, concentration of measure}
\subjclass[2010]{60F10, 05C80, 60C05, 05A20}

\begin{abstract}
We present a general technique for computing large deviations of nonlinear functions of independent Bernoulli random variables.  The method is applied to compute the large deviation rate functions for subgraph counts in sparse random graphs. Previous technology, based on Szemer\'edi's regularity lemma, works only for dense graphs.  Applications are also made to exponential random graphs and three-term arithmetic progressions in random sets of integers. 
\end{abstract}

\maketitle

%\tableofcontents

\section{Introduction}
\subsection{A motivating example}\label{motiv}
Let $G(N,p)$ be the Erd\H{o}s--R\'enyi random graph on $N$ vertices with edge probability $p$, that is, the classical model where any two vertices are connected by an edge with probability $p$, independent of all else. Let $T$ denote the number of triangles in this graph. It has been an open question in the random graph literature for a long time~\cite{JR02} to determine the behavior of the upper tail of $T$, that is, probabilities of the type $\pp(T\ge (1+\delta)\ee(T))$. The main difficulty with this problem, and the reason why it may be appealing to a probabilist, is that the standard tools from concentration of measure and other probability inequalities do not seem to work so well in this setting, in spite of the fact that the number of triangles in an Erd\H{o}s--R\'enyi graph is simply a degree three polynomial of independent Bernoulli random variables. 

After a series of successively improving suboptimal results by many authors over many years, a big advance was made by Kim and Vu \cite{kimvu04} and simultaneously by Janson et al.~\cite{JOR04} in 2004 who showed that if $p\ge N^{-1}\log N$,  then 
\[
\exp(-c_1(\delta) N^2 p^2\log(1/p))\le \pp(T\ge (1+\delta)\ee(T))\le \exp(-c_2(\delta)N^2p^2)\, ,
\]
where $c_1(\delta)$ and $c_2(\delta)$ are constants depending on $\delta$ only. 

Several years later, the logarithmic discrepancy between the exponents on the two sides was removed by Chatterjee~\cite{chatriangle} and independently  by DeMarco and Kahn \cite{dk, dk2}, where it was shown that when $p\ge N^{-1}\log N$,
\begin{align*}
\exp(-c_1(\delta) N^2 p^2\log(1/p))&\le \pp(T\ge (1+\delta)\ee(T))\\
&\le \exp(-c_2(\delta)N^2p^2\log(1/p))\, . 
\end{align*}
This still left open the question of determining the dependence of the exponent on $\delta$. When $p$ is fixed and $N$ tends to infinity, the problem was solved in 2011 by Chatterjee and Varadhan \cite{cv1}, confirming a conjecture from an unpublished manuscript of Bolthausen, Comets and Dembo \cite{bolthausenetal09}. In \cite{cv1}, it was shown that for fixed $p\in (0,1)$ and $\delta >0$,
\begin{equation}\label{motiveq}
 \pp(T\ge (1+\delta)\ee(T)) = \exp(-c(\delta, p)N^2(1-o(1)))
\end{equation}
as $N\ra\infty$, where 
\begin{equation}\label{variational}
c(\delta, p) = \frac{1}{2}\inf_f\{I_p(f): T(f)\ge (1+\delta)p^3\}\, ,
\end{equation}
where $f:[0,1]^2 \ra[0,1]$ is any Lebesgue measurable function that satisfies $f(x,y)=f(y,x)$ for all $x$ and $y$, 
\[
I_p(f) = \iint_{[0,1]^2} \biggl(f(x,y)\log \frac{f(x,y)}{p} + (1-f(x,y)) \log \frac{1-f(x,y)}{1-p}\biggr)\, dx\, dy\, ,
\]
and 
\[
T(f) = \iiint_{[0,1]^3}f(x,y)f(y,z)f(z,x) \, dx\, dy\, dz\, .
\]
Incidentally, the variational problem \eqref{variational} has not yet yielded explicit solutions except in special ranges of $\delta$ and $p$ \cite{chadey, cv1, lz, lz2}. 

The above result was proved using Szemer\'edi's regularity lemma \cite{szemeredi78} from graph theory. A well known problem with Szemer\'edi's lemma is that it yields very poor quantitative bounds, which makes it virtually impossible to extend the arguments of \cite{cv1} to the case where $p$ is allowed to tend to  $0$ as $N\ra\infty$. One can show (e.g.~in \cite{lz2}) that  a weaker version of Szemer\'edi's lemma suffices for the proof in \cite{cv1}, which makes it possible to make the technique work when $p$ tends to zero slower than a negative power of $\log N$, but it seems safe to bet that a Szemer\'edi type argument cannot help when $p$ goes to zero like a negative power of $N$.

The last problem mentioned in the previous paragraph, namely, computing $c(\delta, p)$ in \eqref{motiveq} when $p$ goes to zero like $N^{-\alpha}$ for some $\alpha> 0$, was the original motivation for this paper. What we accomplish in this article is the following: We build a general machinery for tackling large deviations of certain class of nonlinear functions of independent Bernoulli random variables, which in particular 
circumvents the use of Szemer\'edi's lemma. Among other things, this approach yields a variational formula for $c(\delta, p)$, analogous to \eqref{variational}, which 
holds when $p=p(N) \to 0$ slower than $N^{-1/42}$. To see what its potential benefits 
may be, we note that after the first version of this paper was posted on arXiv, 
Lubetzky and Zhao \cite{lz2} found a way to explicitly solve our variational problem 
when $N^{-1}\ll p\ll1$. As a result of this additional exciting development, we now 
know that if $p(N) \to 0$ slower than $N^{-1/42}$, then  
\[
\pp(T\ge (1+\delta)\ee(T)) = \exp\biggl(-(1-o(1)) \min\biggl\{\frac{\delta^{2/3}}{2}\, , \, \frac{\delta}{3}\biggr\}N^2 p^2 \log \frac{1}{p}\biggr)\, .
\] 
Therefore, the variational formula for $c(\delta, p)$ proved in this paper and its solution by Lubetzky and Zhao have completed the quest for understanding the behavior of the upper tail of triangle counts in Erd\H{o}s--R\'enyi random graphs under the restriction that $p\gg N^{-1/42}$. It has been conjectured in~\cite{lz2} that the same formula should hold all the way down to $p\gg N^{-1/2}$. A strong evidence in favor of this conjecture is that the Lubetzky--Zhao solution of the variational formula for $c(\delta, p)$ holds whenever $p\gg N^{-1/2}$. The above result has been recently extended to more general subgraph counts by Bhattacharya et al.~\cite{bglz}. For a survey of these developments and a short overview of the emerging field of large deviations for random graphs, see \cite{cha16}. 

\subsection{Goal of the paper}
Suppose that $f:[0,1]^n \ra \rr$ is a function with some amount of smoothness. Let $p$ be a number in the open interval $(0,1)$, and $Y = (Y_1,\ldots, Y_n)$ be a vector of i.i.d.\ $Bernoulli(p)$ random variables. For $u\in [0,1]$ let 
\begin{equation}\label{ipdef1}
I_p(u) := u\log \frac{u}{p} + (1-u)\log \frac{1-u}{1-p}\, ,
\end{equation}
and for each $x= (x_1,\ldots, x_n)\in [0,1]^n$, define 
\begin{equation}\label{ipdef2}
I_p(x) := \sum_{i=1}^n I_p(x_i)\, .
\end{equation}
For each $t\in \rr$, define
\begin{equation}\label{phidef}
\phi_p(t) := \inf\{I_p(x)\, : \, x\in [0,1]^n \text{ such that } f(x)\ge tn\}\, . 
\end{equation}
We want to investigate conditions under which the following ``upper tail approximation'' is valid:
\begin{equation}\label{mainapprox}
\pp(f(Y) \ge tn ) = \exp(-\phi_p(t) + \text{lower order terms})\, .
\end{equation}
Analogous statements may be similarly formulated if the $Y_i$'s have some distributions other than Bernoulli. 

It is known that such an approximation is valid for  continuous functions of the 
empirical measure of i.i.d.~random variables. This fact forms 
the basis of a big part of modern large deviations theory; see 
\cite{dembozeitouni} and the references therein. Since the empirical measure is a linear function of i.i.d.~random objects (Dirac masses at the sample points), this is a class of linear examples. The main goal of this paper is to establish conditions under which this approximation holds in nonlinear settings.

One challenging example of a nonlinear result of the above type is the recent proof 
in \cite{cv1} that the approximation holds for upper tails of 
subgraph counts in dense Erd\H{o}s-R\'enyi random graphs, 
a result which was later generalized to 
random matrices~\cite{cv2} and exponential random graphs \cite{cd}. The proofs in \cite{cv1, cv2, cd} are, however, rather specialized to the random graph setting. The main tool in these papers is the regularity lemma of Szemer\'edi \cite{szemeredi78} and the graph limit theory of Lov\'asz and coauthors \cite{borgsetal08, borgsetal12, lovaszbook} that builds on the regularity lemma. The unavailability of a suitable ``sparse'' version of Szemer\'edi's lemma makes it impossible to extend the results of \cite{cv1, cv2, cd} to sparse graphs. Serious attempts have been made at formulating a sparse graph limit theory and sparse regularity lemmas \cite{br, bccz1, bccz2}, but it is unlikely that these will provide the precision required for large deviations. The reason is that in all existing  formulations, there is always some assumption about the regularity of the graph structure, and it may not be true that random graphs obey such regularity conditions in the large deviations regime. In graph theoretic terminology, the absence of a ``counting lemma'' for sparse graphs is the main impediment to extending a Szemer\'edi type argument to the sparse case.  

More importantly, ideally one should not need to resort to specialized graph theoretic tools to prove an approximation as simple and basic as \eqref{mainapprox} for an $f$ that may be as uncomplicated as a  polynomial of degree three (e.g.\ number of triangles). %One feels that there must be some ``general truth'' that leads to the approximation \eqref{mainapprox}. 

Our main objective here is to give a general error bound for the approximation \eqref{mainapprox} directly in terms of properties of the function $f$
(as elaborated in the sequel), with an error bound small enough to allow extension of the aforementioned graph theoretic large deviation results to sparse random graphs. %Here ``sparse'' means, for example, an Erd\H{o}s-R\'enyi graph $G(n,p)$ where $p$ is allowed to vanish like $n^{-\alpha}$ for some positive $\alpha$. A na\"ive ``sparsification'' of the existing argument based on the regularity lemma is unlikely to allow $p$ going to zero at a rate much faster than some negative power of $\log^*n$. 
Incidentally, there are several notable results on upper bounds for tail probabilities for nonlinear functions of independent Bernoulli random variables. The bounded difference inequality \cite{mcdiarmid89} has been available for a long time. Improved inequalities 
were discovered by Talagrand \cite{talagrand95}, Lata\l a \cite{latala97}, Kim and Vu \cite{kimvu00} and Vu \cite{vu02}. 
%In particular, the methods of Vu \cite{vu02} are sometimes strong enough to bound the tail probabilities of subgraph counts in sparse 
%random graphs within a universal constant factor multiplying $\phi_p(t)$. 
However, all these methods seem to fall short of 
proving an approximation such as \eqref{mainapprox}.

\subsection{The main result}\label{mainsec}
Our main result is Theorem \ref{uppertailthm} which gives a sufficient condition for the validity of the approximation \eqref{mainapprox}. This 
sufficient condition may be roughly described as follows: The approximation \eqref{mainapprox} is valid when, in addition to some minor smoothness conditions on the function $f$, {\it the gradient vector $\nabla f(x) = (\partial f/\partial x_1, \ldots \partial f/\partial x_n)$ may be approximately encoded by $o(n)$ bits of information}. One may call this the ``low complexity gradient'' condition. 

To illustrate this, consider the simple case of
 \[
f(x) = \sum_{i=1}^{n-1} x_i x_{i+1}\,,
\]
where the approximation~\eqref{mainapprox} is not valid. Indeed, 
large deviation probabilities for this function are related to the one-dimensional Ising model, and easily shown to not satisfy \eqref{mainapprox}. For this function, $\partial f/\partial x_i = x_{i-1} + x_{i+1}$ for $2\le i\le n-1$, $\partial f /\partial x_1= x_2$, and $\partial f/\partial x_n = x_{n-1}$, so clearly this gradient vector cannot be approximately encoded by $o(n)$ many bits; we effectively need to know all the $x_i$'s to encode the gradient vector of this function. 
On the other hand,~if 
\[
f(x)=\frac{1}{n}\sum_{1\le i<j\le n} x_i x_j\, ,
\]
as in the Currie-Weiss model, then for each $i$,
\[
\fpar{f}{x_i}= \frac{1}{n}\sum_{j\ne i} x_j = -\frac{x_i}{n} + \frac{1}{n}\sum_{j=1}^n x_j\, .
\]
Thus, the gradient vector is approximately encoded by the single quantity $n^{-1}\sum x_j$,
so the ``low complexity gradient'' condition holds and the large deviation probabilities 
(which in this trivial case are covered by the general theory of large deviations), 
satisfy \eqref{mainapprox}. 

Unfortunately, although its content matches very well the preceding description, 
the actual statement of the theorem is somewhat messier, and requires some 
additional notation which we introduce next.

Let $\|f\|$ denote the supremum norm of $f:[0,1]^n \ra\rr$. Suppose that
$f:[0,1]^n \ra\rr$ is twice continuously differentiable in $(0,1)^n$, such 
that $f$ and all its first and second order derivatives extend continuously to the boundary. 
For each $i$ and $j$, let
\[
f_i := \fpar{f}{x_i}\ \  \text{ and } \ \ 
f_{ij} := \mpar{f}{x_i}{x_j}. 
\]
Define 
\[
a := \|f\|, \ \ b_i := \|f_i\| \ \ \text{ and } \ \ c_{ij} := \|f_{ij}\|\, . 
\]
Given $\ep>0$, let $\dd(\ep)$ be a finite subset of $\rr^n$ such that for all $x\in \{0,1\}^n$, there exists $d = (d_1,\ldots, d_n)\in \dd(\ep)$ such that 
\begin{equation}\label{entropy}
\sum_{i=1}^n (f_i(x)- d_i)^2\le n \ep^2.
\end{equation}
The following theorem gives an error bound for the approximation \eqref{mainapprox} in terms of the quantities $a$, $b_i$, $c_{ij}$ and the sizes of the sets $\dd(\ep)$. 
\begin{thm}\label{uppertailthm}
For $f$ as  above, $p\in (0,1)$ and $Y$ a vector of $n$ i.i.d.\ $Bernoulli(p)$ random variables, let $\phi_p$ be defined as in \eqref{phidef}.
Then, for any $\delta>0$, $\ep>0$ and $t\in \rr$,
\begin{align*}
\log \pp(f(Y)\ge tn) &\le - \phi_p(t-\delta) + \textup{complexity term} \\
&\qquad + \textup{smoothness term}\, ,
\end{align*}
where with $a$, $b$, $c_{ij}$, $\dd(\ep)$ defined above, and $K:= \phi_p(t)/n$, 
\begin{align*}
\textup{complexity term} &:= 
\frac{1}{4}\Big(n\sum_{i=1}^n\beta_i^2\Big)^{1/2}\ep + 3n\ep +\log \biggl(\frac{4K(\frac{1}{n}\sum_{i=1}^n b_i^2)^{1/2}}{\delta\ep}\biggr) \\
&\qquad + \log |\dd((\delta\ep)/(4K))|\, , \ \text{ and}\\
%\end{align*}
%and
%\begin{align*}
\textup{smoothness term} &:= 4\biggl(\sum_{i=1}^n(\alpha \gamma_{ii} + \beta_i^2)+ \frac{1}{4}\sum_{i,j=1}^n \bigl(\alpha \gamma_{ij}^2+\beta_i \beta_j \gamma_{ij} + 4\beta_i \gamma_{ij}\bigr)\biggr)^{1/2} \\
&\qquad + \frac{1}{4}\Big(\sum_{i=1}^n \beta_i^2\Big)^{1/2}\Big(\sum_{i=1}^n \gamma_{ii}^2\Big)^{1/2} + 3\sum_{i=1}^n \gamma_{ii}+ \log 2\, ,
\end{align*}
for
\begin{align*}
\alpha &:= nK + n|\log p| + n|\log (1-p)|\, ,\\
\beta_i &:= \frac{2K b_i}{\delta} + |\log p| + |\log(1-p)|\, , \text{ and}\\
\gamma_{ij}  &:= \frac{2K c_{ij}}{\delta} +\frac{ 6K b_i b_j}{n\delta^2}\, .
\end{align*}
Moreover,  
\begin{align*}
\log \pp(f(Y)\ge tn) &\ge -\phi_p(t+\delta_0) - \ep_0 n - \log 2\, ,
\end{align*}
where
\[
\ep_0  := \frac{1}{\sqrt{n}}\biggl(4+\biggl|\log\frac{p}{1-p}\biggr|\biggr)
\]
and
\[
\delta_0 := \frac{2}{n}\biggl(\sum_{i=1}^n (ac_{ii} + b_i^2)\biggr)^{1/2}\, .
\]
\end{thm}
We do not attempt to produce a watered down cleaner error bound, since the full power of Theorem \ref{uppertailthm} is needed in our applications.
 
%While it is possible that there are examples where the error bounds provided by Theorem \ref{uppertailthm} are not useful but the approximation \eqref{mainapprox} is valid, we have some examples to demonstrate that the theorem may be useful in nontrivial situations. %, and may even lead to solutions of long-standing open problems. 

\subsection{Application to subgraph counts}\label{subgraphsec}
Let $G = G(N,p)$ be an Erd\H{o}s-R\'enyi random graph on $N$ vertices, with edge probability $p$. Let $H$ be a fixed finite simple graph. Let $\hom(H,G)$ be the number of homomorphisms (edge-preserving maps) from the vertex set $V(H)$ of $H$ into the vertex set $V(G)$ of $G$. This is slightly different than the number of copies of $H$ in $G$, but nicer to work with mathematically. The ``homomorphism density'' of $H$ in $G$ is defined as 
\[
t(H, G) := \frac{\hom(H,G)}{|V(G)|^{|V(H)|}}\, .
\]
Our object of interest is the large deviation rate function for the upper tail of $t(H,G)$.  Let $\mathcal{P}$ denote upper triangular arrays like $x= (x_{ij})_{1\le i<j\le N}$, where each $x_{ij}\in [0,1]$. For any $x\in \mathcal{P}$, let $G_x$ denote the undirected random graph whose edges are independent, and edge $\{i,j\}$ is present with probability $x_{ij}$, and absent with probability $1-x_{ij}$. Let $t(H,x)$ denote the expected value of $t(H, G_x)$. Explicitly, if $H$ has vertex set $\{1,2,\ldots,k\}$ and edge set $E(H)$, then
\[
t(H, x) = \frac{1}{N^{k}}\sum_{q_1,\ldots, q_k=1}^N\prod_{\{l,l'\}\in E(H)} x_{q_lq_{l'}}\, ,
\]
where $x_{ii}$ is interpreted as zero for each $i$ and $x_{ji}=x_{ij}$. For $x\in \mathcal{P}$, define
\[
I_p(x) := \sum_{1\le i<j\le N} I_p(x_{ij}),
\]
where $I_p(x_{ij})$ is defined as in \eqref{ipdef1}. 
For each $u > 1$ define 
\[
\psi_p(u) := \inf\{I_p(x): x\in \mathcal{P} \text{ such that } t(H,x)\ge u\, \ee(t(H,G))\}\,.
\]
The following theorem shows that  for any $u>1$, 
\begin{align}
\pp\bigl(t(H,G)\ge u\, \ee(t(H,G))\bigr) &= \exp(-\psi_p(u)+\text{lower order terms})\, , \label{approx2}
\end{align}
provided that $N$ is large and $p$ is not too small. This approximation was proved for fixed $p$ and $N$ growing to infinity in \cite{cv1} using Szemer\'edi's lemma. Various interesting consequences of this variational formula were proved in~\cite{cv1, lz}. 
\begin{thm}\label{graphthm}
Take any finite simple graph $H$ and let $t(H,G)$ and $\psi_p$ be defined as above. Let $k$ be the number of vertices of $H$, $m$ be the number of edges of $H$, and $\Delta$ be the maximum degree of $H$. Let $X:= t(H,G)$. Suppose that $m\ge 1$ and $N^{-1/(m+3)}\le p \le 1-N^{-1}$. Then for any $u> 1$ and any $N$ sufficiently large (depending only on $H$ and $u$), 
\begin{align*}
1-\frac{c (\log N)^{b_1}}{N^{b_2} p^{b_3}} \le \frac{\psi_p(u)}{-\log \pp(X\ge u\, \ee(X))} \le 1+ \frac{C (\log N)^{B_1}}{N^{B_2} p^{B_3}} \, ,
\end{align*}
where $c$ and $C$ are constants that depend only on  $H$ and $u$, and 
\begin{align*}
&b_1 = 1\, , \  \  b_2 = \frac{1}{2m}\, , \ \ b_3 = \Delta\, ,\\
&B_1 = \frac{9+8m}{5+8m}\, , \  \  B_2 = \frac{1}{5+8m}\, , \ \ B_3 = \Delta - \frac{16m}{k(5+8m)}\, .
\end{align*}
\end{thm}
For example, when $H$ is a triangle, an explicit computation of the  error terms shows that the approximation \eqref{approx2} holds whenever $p$ goes to zero at a rate slower than $N^{-1/42}(\log N)^{11/14}$. There is no reason to believe that this should be the optimal threshold for the validity of the approximation~\eqref{approx2}, but  at least it allows a polynomial rate of decay for $p$.

Shortly after the first draft of this paper was put up on arXiv, Lubetzky and Zhao \cite{lz2} explicitly computed by a remarkably clever argument the limiting behavior of $\psi_p(u)$ when $H$ is a triangle and $N^{-1}\ll p\ll 1$. With the aid of Theorem \ref{graphthm}, this completely solves the large deviation problem for triangle counts when $N^{-1/42}(\log N)^{11/14} \ll p\ll 1$. Combining the solution of Lubetzky and Zhao with Theorem \ref{graphthm} gives the following result:
\begin{thm}[Lubetzky and Zhao \cite{lz2}]\label{lzcor}
Let $T$ be the number of triangles in $G(N,p)$. Then for any fixed $\delta  > 0$,
\[
\pp(T\ge (1+\delta)\ee(T)) = \exp\biggl(-(1+o(1)) \min\biggl\{\frac{\delta^{2/3}}{2}\, , \, \frac{\delta}{3}\biggr\}N^2 p^2 \log \frac{1}{p}\biggr)
\]
when $N\ra \infty$ and $p\ra 0$, subject to the constraint that $p\ge N^{-1/42}\log N$. 
\end{thm}
As mentioned above, \cite[Theorem 1.1]{lz2} gives the explicit limiting behavior of $\psi_p(u)$ whenever $p$ goes to zero at a rate slower than $N^{-1}$. Therefore if one can prove a version of Theorem~\ref{graphthm} that allows $p$ to decay like $N^{-1+\ep}$, that would solve the problem of large deviations for triangle counts in its entirety.

More recently, Theorem \ref{lzcor} has been generalized by Bhattacharya et al.~\cite{bglz}, who got the following beautiful result by analyzing the variational formula of Theorem \ref{graphthm}. Take any finite simple graph $H$ with maximum degree~$\Delta$. Let $H^*$ be the induced subgraph of $H$ on all vertices whose degree in $H$ is $\Delta$. Recall that an independent set in a graph is a set of vertices such that no two are connected by an edge. Also, recall that a graph is called regular if all its vertices have the same degree, and irregular otherwise. Define a polynomial 
\[
P_{H^*}(x) := \sum_k i_{H^*}(k) x^k\,,
\] 
where $i_{H^*}(k)$ is the number of $k$-element independent sets in $H^*$. The main result of \cite{bglz} is the following. %Theorem \ref{lzthm} is a special case of this result. 
\begin{thm}[Bhattacharya et al.~\cite{bglz}]\label{bthm}
Let $H$ be a connected finite simple graph on $k$ vertices with maximum degree $\Delta \ge 2$. Then for any $\delta >0$, there is a unique positive number $\theta = \theta(H, \delta)$ that solves $P_{H^*}(\theta)=1+\delta$, where $P_{H^*}$ is the polynomial defined above. Let $H_{N,p}$ be the number of homomorphisms of $H$ into a $G(N,p)$ random graph. Then there is a constant $\alpha_H>0$ depending only on $H$, such that if $N\to \infty$ and $p\to 0$ slower than $N^{-\alpha_H}$, then for any~$\delta >0$,
\[
\pp(H_{N,p} \ge (1+\delta) \ee(H_{N,p})) = \exp\biggl(-(1+o(1))c(\delta) N^2 p^\Delta \log \frac{1}{p}\biggr)\,,
\]
where 
\[
c(\delta) = 
\begin{cases}
\min\{\theta,\frac{1}{2}\delta^{2/k}\} & \text{ if $H$ is regular,} \\
\theta &\text{ if $H$ is irregular.}
\end{cases}
\]
%Here $|V(H)|$ and $|E(H)|$ stand for the number of vertices and the number of edges in $H$. 
\end{thm}
The formula given in Theorem \ref{bthm} is more than just a formula. It gives a hint at the conditional structure of the graph, and at the nature of phase transitions as $\delta$ varies. Unlike the dense case, it is hard to give a precise meaning to claims about the conditional structure in the sparse setting due to the lack of an adequate sparse graph limit theory. For a detailed discussion, see \cite{bglz, cha16}.

The paper \cite{bglz} also gives a number of examples where the coefficient $c(\delta)$ in Theorem \ref{bthm} can be explicitly computed. For instance, if $H = C_4$, the cycle of length four, then
\[
c(\delta) = 
\begin{cases}
\frac{1}{2}\sqrt{\delta} &\text{ if } \delta<16,\\
-1+\sqrt{1+\frac{1}{2}\delta} &\text{ if } \delta \ge 16.
\end{cases}
\]
Theorem \ref{lzcor} is also a special case of Theorem \ref{bthm}.

The proof of Theorem \ref{graphthm} is a direct application of Theorem \ref{uppertailthm}. The main challenge lies in verifying the low complexity gradient condition. In the case of dense graphs, the condition may be verified using Szemer\'edi's lemma. But it turns out that Szemer\'edi's lemma is not a strict requirement for proving the low complexity gradient condition for subgraph counts. One can bypass that and use a spectral argument instead. The spectral argument generalizes easily to the sparse case. 

Incidentally, as already discussed in Subsection \ref{motiv}, 
%a lot is known about 
the rough order of probability upper tails for subgraph counts 
drew significant interest in the random graphs community for 
a long time (as indicated in \cite{JR02}). It was eventually determined in a series of papers by Vu \cite{vu01, vu02}, Kim and Vu \cite{kimvu00, kimvu04}, Janson and Ruci\'nski \cite{JR04} and finally by Janson, Oleszkiewicz and Ruci\'nski \cite{JOR04}. The upper and lower bounds obtained by these authors differed by a logarithmic factor; they were matched in \cite{chatriangle, dk} for triangle counts, and for counts of cliques in \cite{dk2}. The techniques of all of these papers, however, are only suitable for getting the tail 
decay order and a first-order approximation such as the one given in 
Theorem \ref{graphthm} is not achievable by these methods.

\subsection{Application to arithmetic progressions}
Fixing $n \in \N$ and $p \in (0,1)$, let $A$ be a random subset of $\zz/n\zz$, constructed by keeping each element with probability $p$, and dropping with probability $1-p$. In this subsection we apply Theorem \ref{uppertailthm} to compute large deviation probabilities for the number of three-term arithmetic progressions in $A$. One may be able to tackle longer arithmetic progressions via Theorem \ref{uppertailthm}, but this would 
require finding a better upper bound on its complexity term. 
\begin{thm}\label{arith}
Let $A$ be a random subset of $\zz/n\zz$, constructed as above. Let $X$ be the number of pairs $(i,j)\in (\zz/n\zz)^2$ such that $\{i,i+j, i+2j\}\subseteq A$. Let $I_p$ be defined as in \eqref{ipdef2} and define
\begin{align*}
\theta_p(u) := \inf\biggl\{&I_p(x): x\in [0,1]^{\zz/n\zz} \\
&\qquad  \textup{such that $\sum_{i,j\in \zz/n\zz} x_i x_{i+j} x_{i+2j}\ge u\, \ee(X)$}\biggr\}\, .
\end{align*}
Suppose that $n^{-1/162}\le p\le1-n^{-1}$. Then for any $u>1$,
\begin{align*}
1 - c\, n^{-1/6}p^{-6}\log n&\le \frac{\theta_p(u)}{-\log \pp(X\ge u\, \ee(X))} \le 1 + Cn^{-1/29}p^{-162/29}(\log n)^{33/29}\, ,
\end{align*}
where $C$ and $c$ are constants that may depend only on $u$.
\end{thm}
This theorem gives an approximation for the upper tail of the number of three-term arithmetic progressions in random subsets of $\zz/n\zz$, even when the random subset is allowed to be somewhat sparse ($p\gg n^{-1/162}(\log n)^{33/162}$). Note that with $p=1/2$, the upper tail probability is proportional to the number of subsets of $\zz/n\zz$ that contain more than a given number of three-term progressions.

Again, the main challenge in the proof of Theorem \ref{arith} is in establishing the low complexity gradient condition. Discrete Fourier transform techniques are used to prove that this condition holds for the function $f$ defined above. We believe that the low complexity gradient condition should apply for longer arithmetic progressions, as it may be expected to hold in any situation where some kind of ``averaging'' is going on; if true, this would
extend our solution to longer progressions. 

The study of arithmetic progressions in subsets of integers has a long and storied history, most of which is concerned with questions of existence. An excellent survey of old and new results is available in Tao and Vu \cite{taovu}. Counting the number of sets with a given number of arithmetic progressions, or understanding the typical structure of sets that contain lots of progressions, are challenges of a different type, falling within the purview of large deviations theory. Recently a certain amount of interest has begun to grow around the resolution of such questions, quickly leading to the realization that conventional large deviations theory will not provide the answers. The most pertinent papers are the recent articles on probabilistic properties of the so-called ``non-conventional averages'' by Kifer \cite{kifer}, Kifer and Varadhan~\cite{kv, kv2} and Carinci et al.~\cite{carinci}. In particular, Carinci et al.~\cite{carinci} prove a large deviation principle for what they call ``two-term arithmetic progressions'', which are sums of the type $\sum x_i x_{2i}$.

\subsection{Approximation of normalizing constants}\label{normalizing}
Let $f$ be as in Subsection \ref{mainsec}. Consider a probability measure on $\{0,1\}^n$ that puts mass proportional to $e^{f(x)}$ at each point $x$. The logarithm of the normalizing constant of this probability measure, sometimes called the ``free energy'', is
\[
F := \log  \sum_{x\in \{0,1\}^n} e^{f(x)}\, .
\]
The free energy is an important object in statistical physics. In this context, the probability measure defined above is called the ``Gibbs measure'' with Hamiltonian $f$. The free energy encodes useful information about the structure of the Gibbs measure: it is often used to compute the Gibbs averages of various quantities of interest by differentiating the free energy with respect to appropriate parameters.  Computation of normalizing constants is also important in statistics because it is required for computing  maximum likelihood estimates of unknown parameters.

For $u\in [0,1]$, define
\[
I(u) := u\log u + (1-u)\log (1-u)\, .
\]
For $x = (x_1,\ldots, x_n)\in [0,1]^n$, let
\[
I(x) := \sum_{i=1}^n I(x_i)\, .
\]
The goal of this subsection is to investigate conditions on $f$ under which the approximation
\begin{align*}
F = \sup_{x\in [0,1]^n} (f(x)-I(x)) + \text{lower order terms}
\end{align*}
is valid. As expected from the general connection between large deviations and moment generating functions given by Varadhan's lemma (see in \cite{dembozeitouni}), the validity of the above approximation is closely related to that of \eqref{mainapprox}. Incidentally, it is easy to verify that exact equality holds in the above display (without any lower order correction terms) if $f$ is linear.
\begin{thm}\label{freethm}
Let $F$ be defined as above, and $a$, $b_i$, $c_{ij}$ and $\dd(\ep)$ be as in Theorem \ref{uppertailthm}. Then for any $\ep>0$,
\begin{align*}
F &\le \sup_{x\in[0,1]^n}(f(x)-I(x)) + \textup{complexity term} + \textup{smoothness term}\, ,
\end{align*}
where
\begin{align*}
\textup{complexity term} &= \frac{1}{4}\Big(n\sum_{i=1}^nb_i^2\Big)^{1/2}\ep + 3n\ep +\log |\dd(\ep)|, \ \text{ and}\\
%\end{align*}
%and
%\begin{align*}
\textup{smoothness term} &= 4\biggl(\sum_{i=1}^n(ac_{ii} + b_i^2)+ \frac{1}{4}\sum_{i,j=1}^n \bigl(ac_{ij}^2+b_i b_j c_{ij} + 4b_i c_{ij}\bigr)\biggr)^{1/2} \\
&\qquad + \frac{1}{4}\Big(\sum_{i=1}^n b_i^2\Big)^{1/2}\Big(\sum_{i=1}^n c_{ii}^2\Big)^{1/2} + 3\sum_{i=1}^n c_{ii}+ \log 2.
\end{align*}
Moreover, $F$ satisfies the lower bound
\[
F \ge \sup_{x\in [0,1]^n} (f(x)-I(x)) - \frac{1}{2}\sum_{i=1}^n c_{ii}\, .
\]
\end{thm}
Just like Theorem \ref{uppertailthm}, it is unlikely that the error terms in Theorem \ref{freethm} are sharp. Still, it is the first result of its kind and good enough to be applicable in some examples of interest.

Actually, Theorem \ref{uppertailthm} is proved in this paper as a special application Theorem \ref{freethm}. To see how this is done, take a function $g:[0,1]^n \ra\rr$ and a threshold $t\in \rr$. Let $f$ be a smooth function such that
\[
f(x) =
\begin{cases}
0 & \text{ if } g(x)\ge tn\, ,\\
\text{a large negative number} &\text{ if } g(x)< tn - \text{a small quantity.}
\end{cases}
\]
Then $e^{f(x)}$ is a smooth approximation to the function that is $1$ when $g(x)\ge tn$ and $0$ when $g(x)< tn$. One may now try to apply Theorem \ref{freethm} with this $f$ to find an approximation to $\pp(g(Y)\ge tn)$. This strategy is similar to the one used in Bryc's proof of the inverse Varadhan lemma (see \cite[Section 4.4]{dembozeitouni}). 

The usual large deviation technique of obtaining optimal upper bounds using moment generating functions does not seem to work for sparse random graphs. This is the reason why the above scheme is needed for deriving Theorem \ref{uppertailthm} from Theorem \ref{freethm}. This is also the main reason why the error bound in Theorem~\ref{uppertailthm} is somewhat lossy, leading to the suboptimal conditions on the decay rate of $p$ in Theorems \ref{lzcor} and~\ref{bthm}.

The above sketch indicates that Theorem \ref{freethm} is much more general than Theorem \ref{uppertailthm}. Indeed, using a similar tactic it may be used for computing joint large deviations for several functions simultaneously, although we will not pursue this direction here. 

\subsection{Application to exponential random graphs}
In this section we will use the notation of Subsection \ref{subgraphsec}. Let $l$ be a positive integer and $H_1,\ldots, H_l$ be finite simple graphs. Let $\beta_1,\ldots, \beta_l$ be $l$ real numbers. Let $N$ be another positive integer. Given a simple graph $G$ on $N$ vertices, let $t(H, G)$ denote, as in Subsection \ref{subgraphsec}, the homomorphism density of $H$ in $G$. 

Consider the probability measure on the set of all simple graphs on $N$ vertices that puts mass proportional to 
\[
\exp\bigl(N^2(\beta_1 t(H_1, G)+\cdots + \beta_l t(H_l, G))\bigr)
\]
on each graph $G$. This is an example of an exponential random graph model (ERGM). Such models are widely used in the statistical social networks community to understand the structure of networks. One of the key objectives of the practitioners is to compute estimates of the parameters $\beta_1,\ldots,\beta_l$ from an observed graph, which they assume is drawn from this model. The most popular approach to estimation is the maximum likelihood method. To implement this method, however, one needs to know the normalizing constant of the probability measure. 

Until recently, the only available techniques for approximating the normalizing constants of such probability measures all relied on Markov Chain Monte Carlo (MCMC) methods. There are some doubts about the accuracy of such approximations, as pointed out in \cite{bbs}.  The mathematical problem was solved in \cite{cd} where it was shown that if $Z_N$ is the normalizing constant, then as $N$ goes to infinity (keeping $\beta_1,\ldots, \beta_l$ fixed), 
\[
\frac{\log Z_N}{N^2} \approx \sup_{x\in \mathcal{P}_N} \biggl(\beta_1t(H_1, x)+\cdots+\beta_l t(H_l, x) - \frac{I(x)}{N^2}\biggr) =: L_N\, ,
\]
where $\mathcal{P}_N$ denotes the set $\mathcal{P}$ defined in Subsection \ref{subgraphsec}, that is, the set of all $x = (x_{ij})_{1\le i<j\le N}$ with $x_{ij}\in [0,1]$ for all $i,j$. Here the approximation sign means that the difference between the two sides tends to zero. The proof of this theorem is based on the large deviation principle for Erd\H{o}s-R\'enyi graphs from \cite{cv1}. Since this argument is based on Szemer\'edi's lemma, it does not give error bounds that are better than some negative power of $\log^* N$. Another problem is that this result does not allow varying the $\beta$'s with $N$, making it inapplicable for sparse exponential random graphs. 

Theorem \ref{freethm} solves both problems to a certain extent, by giving a concrete error bound.
\begin{thm}\label{ergmthm}
Let $Z_N$ and $L_N$ be as above. Let $B := 1+|\beta_1|+\cdots +|\beta_l|$. Then 
\begin{align*}
-cBN^{-1}&\le \frac{\log Z_N}{N^2} - L_N \\
&\le CB^{8/5} N^{-1/5}(\log N)^{1/5}\biggl(1+\frac{\log B}{\log N}\biggr) + C B^2 N^{-1/2}\, ,
\end{align*}
where $C$ and $c$  are  constants that may depend only on $H_1,\ldots, H_l$.
\end{thm} 
As an example, consider the case where $l=2$, $H_1$ is a single edge, and $H_2$ is a triangle. In this case the above theorem shows that the difference between $N^{-2}\log Z_N$ and $L_N$ tends to zero as long as $|\beta_1|+|\beta_2|$ grows slower than $N^{1/8}(\log N)^{-1/8}$, thereby allowing a small degree of sparsity. When the $\beta$'s are fixed, it provides an approximation error bound of order $N^{-1/5}(\log N)^{1/5}$, substantially better than the negative powers of $\log^*N$ given by Szemer\'edi's lemma. However, the error bound is probably suboptimal. It is an interesting challenge to figure out a sharp error bound.

As mentioned in the previous subsection, it is in general not possible to pass from estimates for exponential random graphs to large deviations for Erd\H{o}s--R\'enyi graphs by optimizing over the parameters. In fact, this is not possible even in the dense setting. The reason is that exponential random graphs have discontinuous phase transitions --- as the parameters vary, the structure of the graph changes abruptly, missing out a range of intermediate structures.  For details, see \cite{cd}.

\subsection{Open problems}
The following is a partial list of questions that are currently beyond the reach of the theory presented in this manuscript, but may be solvable by a more refined theory.
\begin{enumerate}%[1.]
\item Improve Theorem \ref{uppertailthm}, so that results like Theorem~\ref{lzcor} and Theorem \ref{bthm} can be proved when $p$ tends to zero at an optimal rate.
\item As an example of the above, show that Theorem~\ref{lzcor} holds when $p\ra 0$ slower than $n^{-1/2}$.
\item Develop a sparse regularity lemma and a sparse graph limit theory that is powerful enough to prove results like Theorem~\ref{lzcor}. In fact, a reasonable test for the completeness of a sparse graph limit theory is whether it can lead to a solution of the large deviation question for sparse Erd\H{o}s--R\'enyi random graphs. This is because analyzing the large deviation behavior of $G(N,p)$ for small $p$ requires a full understanding of {\it all} possible sparse graph structures rather than focusing a small subset of graphs with nice properties.
\item Extend the large deviation results for three-term arithmetic progressions (Theorem \ref{arith}) to longer progressions. In this paper, discrete Fourier analysis is used for the analysis of three-term progressions. The method does not seem to extend easily to longer progressions. It is possible that higher order Fourier analysis (Gowers norms) or a sparse hypergraph regularity lemma may be needed for longer progressions.
\item Find explicit solutions to the variational problems coming from arithmetic progressions, in the spirit of Theorems~\ref{lzcor} and \ref{bthm}.
\item Improve the result for exponential random graphs (Theorem \ref{ergmthm}) so that sparser graphs can be handled.
\end{enumerate}

\section{Proof sketch}
In this section we give a sketch of the main ideas behind the proof of Theorem \ref{freethm} and the main ideas behind the proof of the low complexity gradient condition for subgraph counts (which is the key ingredient in the proof of Theorem \ref{graphthm}). Note that we have already sketched how Theorem \ref{uppertailthm} follows from Theorem \ref{freethm} in Subsection \ref{normalizing}.

We will generally denote the $i$th coordinate of a vector $x\in \rr^n$ by $x_i$. Similarly, the $i$th coordinate of a random vector $X$ will be denoted by $X_i$. 

Let $X=(X_1,\ldots, X_n)$ be a random vector that has probability density proportional to $e^{f(x)}$ on $\{0,1\}^n$ with respect to the counting measure.  For each $i$, define a function $\hx_i:[0,1]^n \ra [0,1]$ as
\[
\hx_i(x) = \ee(X_i\mid X_j = x_j, \, 1\le j\le n, \ j\ne i). 
\]
Let $\hx:[0,1]^n \ra[0,1]^n$ be the vector-valued function whose $i$th coordinate function is $\hx_i$. 

Let $\HX = \hx(X)$. The first step in the proof is to show that if the smoothness term in Theorem~\ref{freethm} is small, then 
\begin{equation}\label{main1}
f(X) \approx f(\HX) \ \text{ with high probability.}
\end{equation}
(We will not bother to make precise the meaning of $\approx$ in this sketch.) 
To show this, define $D := f(X)-f(\HX)$ and
\[
h(x) := f(x) - f(\hx(x)),
\]
so that $D = h(X)$. For $t\in [0,1]$ and $x\in [0,1]^n$, let 
\begin{align*}
u_i(t, x) &:= f_i(tx +(1-t)\hx(x)), 
\end{align*}
so that
\[
h(x) = \int_0^1 \sum_{i=1}^n (x_i-\hx_i(x)) u_i(t,x)\, dt. 
\]
Thus,
\begin{align}\label{d2o}
\ee(D^2) &= \int_0^1 \sum_{i=1}^n \ee((X_i-\HX_i) u_i(t, X) D) \, dt\, . 
\end{align}
Let $X^{(i)}$ denote the random vector $(X_1,\ldots, X_{i-1}, 0, X_{i+1}, \ldots, X_n)$. Let $D_i := h(X^{(i)})$. Then note that $ u_i(t, X^{(i)}) D_i$ is a function of the random variables $(X_j)_{j\ne i}$ only. Therefore by the definition of $\HX_i$, 
\[
\ee((X_i-\HX_i) u_i(t, X^{(i)}) D_i) = 0. 
\]
Thus, 
\begin{align*}
&\ee((X_i-\HX_i) u_i(t, X) D) \\
&= \ee((X_i-\HX_i) u_i(t, X) D)  - \ee((X_i-\HX_i) u_i(t, X^{(i)}) D_i)\, . 
\end{align*}
If the smoothness term is small, then one can show that $u_i(t, X) \approx u_i(t, X^{(i)})$ and $D \approx D_i$. Therefore, the left-hand side of the above identity is close to zero. By \eqref{d2o}, this proves the approximation \eqref{main1}. 

Define a function $g:[0,1]^n \times [0,1]^n \ra\rr$ as 
\[
g(x,y) := \sum_{i=1}^n (x_i \log y_i + (1-x_i)\log (1-y_i))\, .
\]
By a similar argument as above, it is possible to show that if the smoothness term is small, then 
\begin{equation}\label{main2}
g(X, \HX) \approx g(\HX, \HX) = I(\HX)\, .
\end{equation}
Armed with \eqref{main1} and \eqref{main2}, the proof of Theorem \ref{freethm} may be completed as follows. Let $A$ be the set of all $x$ where $f(x) \approx f(\hx(x))$ and $g(x, \hx(x)) \approx I(\hx(x))$.  By \eqref{main1} and \eqref{main2}, $X\in A$ with high probability. 
That is, 
\[
\frac{\sum_{x\in A} e^{f(x)}}{\sum_{x\in \{0,1\}^n} e^{f(x)}} \approx 1\, .
\]
Therefore by the definition of the set $A$,
\begin{align}
F &= \log \sum_{x\in \{0,1\}^n} e^{f(x)}\approx \log \sum_{x\in A} e^{f(x)}  \approx \log \sum_{x\in A} e^{f(\hx(x)) -I(\hx(x))+g(x, \hx(x))}\, .\label{fineq1o}
\end{align}
The above display is the key to the proof of Theorem \ref{freethm}. It was pointed out to us by Alex Zhai that one way to understand this approximation is to see $f(\hx(x)) -I(\hx(x))+g(x, \hx(x))$ as an approximately piecewise linear proxy for  $f(x)$.

Now let $\ep$ be a small positive number, close to zero. Using the set $\dd(\ep)$, it is easy to produce a set $\dd'(\ep)\subseteq [0,1]^n$ such that $|\dd(\ep)|=|\dd'(\ep)|$, and $\dd'(\ep)$ is an $\ep$-net for the image of $[0,1]^n$ under the map $\hx$. That is, for each $x$ there exists $p\in \dd'(\ep)$ such that
\[
\sum_{i=1}^n (\hx_i(x)-p_i)^2 \le \ep^2 n\, .
\]
We will say that $\hx(x)\approx p$.  For each $p\in \dd'(\ep)$ let $\mathcal{P}(p)$ be the set of all $x\in \{0,1\}^n$ such that $\hx(x)\approx p$. Then  
\begin{align}
&\log \sum_{x\in A} e^{f(\hx(x))-I(\hx(x)) + g(x, \hx(x) )} \label{fineq2o}\\
&\le \log \sum_{p\in \dd'(\ep)}\sum_{x\in \mathcal{P}(p)} e^{f(\hx(x))-I(\hx(x)) + g(x, \hx(x))} \nonumber \\
&\approx \log \sum_{p\in \dd'(\ep)}\sum_{x\in \mathcal{P}(p)} e^{f(p)-I(p) + g(x,p)}. \nonumber
\end{align}
The crucial observation is that for any $p\in [0,1]^n$, 
\[
\sum_{x\in \{0,1\}^n} e^{g(x,p)} = 1. 
\]
Thus, 
\begin{align}
 \log \sum_{p\in \dd'(\ep)}\sum_{x\in \mathcal{P}(p)} e^{f(p)-I(p) + g(x,p)} &\le    \log \sum_{p\in \dd'(\ep)} e^{f(p)-I(p)}\label{fineq3o}\\
 &\le \log|\dd'(\ep)| + \sup_{p\in [0,1]^n} (f(p)-I(p)). \nonumber
\end{align}
Combining \eqref{fineq1o}, \eqref{fineq2o} and \eqref{fineq3o} completes the proof sketch for the upper bound in Theorem \ref{freethm}. 

The proof of the lower bound may be sketched as follows. Take any $y\in [0,1]^n$. Let $Y = (Y_1,\ldots, Y_n)$ be a random vector with independent components, where $Y_i$ is a $Bernoulli(y_i)$ random variable. Then by Jensen's inequality,
\begin{align*}
\sum_{x\in \{0,1\}^n} e^{f(x)} &= \sum_{x\in \{0,1\}^n} e^{f(x)-g(x,y)+g(x,y)}\\
&= \ee(e^{f(Y)-g(Y,y)})\\
&\ge \exp(\ee(f(Y) - g(Y,y)))\\
&= \exp(\ee(f(Y)) - I(y))\, .
\end{align*}
Then, by the same line of argument that is used to prove \eqref{main1} and \eqref{main2}, one can prove that if the error term in the lower bound is small, then $\ee(f(Y)) \approx f(y)$. Since this is true for any $y$, this completes the sketch of the proof of the lower bound.

Our final task in this section is to give a sketch of the proof of the low complexity gradient condition for subgraph counts. For simplicity of exposition, let us just consider the count of triangles. Let $n=N(N-1)/2$ and let us agree to denote elements of $\rr^n$ as $x = (x_{ij})_{1\le i<j\le N}$, with the convention that $x_{ii}=0$ and $x_{ji}=x_{ij}$. Define a function $f:\rr^n \ra \rr$ as 
\[
f(x)=\frac{1}{N}\sum_{i,j,k=1}^Nx_{ij}x_{jk}x_{ki}\, .
\]
Then note that
\[
\fpar{f}{x_{ij}} = \frac{3}{N}\sum_{k=1}^N x_{ik}x_{jk} =: 3a_{ij}(x)\, .
\]
We will now sketch why the numbers $a_{ij}(x)$ may be encoded by $o(N^2)$ bits. For any $x$, let $M(x)$ be the square matrix whose $(i,j)$th entry is $x_{ij}$. Note that for any $x$ and $y$,
\begin{align*}
\sum_{i,j=1}^N (a_{ij}(x)-a_{ij}(y))^2  &= \frac{1}{N^2}\sum_{i,j,k,l} (x_{ik}x_{jk}-y_{ik}y_{jk})(x_{il}x_{jl}-y_{il}y_{jl})\, .
\end{align*}
Let us now expand out the right-hand side and consider one pair of terms:
\begin{align*}
\frac{1}{N^2}\sum_{i,j,k,l} (x_{ik}x_{jk}x_{il}x_{jl}-x_{ik}x_{jk}y_{il}y_{jl})\, .
\end{align*}
This term may be written in a telescoping manner as 
\begin{align*}
\frac{1}{N^2}\sum_{i,j,k,l} x_{ik}x_{jk}x_{il}(x_{jl}-y_{jl}) + \frac{1}{N^2}\sum_{i,j,k,l} x_{ik}x_{jk}(x_{il}-y_{il})y_{jl}\, .
\end{align*}
Let us consider the first term above. The crucial observation is that if $i$ and $k$ are fixed, then the sum in $j$ and $l$ is a quadratic form of the matrix $M(x)-M(y)$. Upon observing this, it is easy to see that this term is bounded above by 
\[
N\|M(x)-M(y)\|_{\text{op}}\, ,
\]
where $\|M(x)-M(y)\|_{\text{op}}$ is the $L^2$ operator norm of the matrix $M(x)-M(y)$. A similar bound may be obtained for all other terms, leading to the conclusion that 
\begin{equation}\label{mainmatrix}
\sum_{i,j} (a_{ij}(x)-a_{ij}(y))^2 \le CN\|M(x)-M(y)\|_{\text{op}}\, ,
\end{equation}
where $C$ is a universal constant. 

Now take any $x$ and let $\lambda_1, \lambda_2, \ldots, \lambda_n$ be the eigenvalues of the symmetric matrix $M(x)$, arranged in decreasing order of magnitude. Then 
\[
\sum_{i=1}^n \lambda_i^2 = \text{Trace}(M(x)^2) = \sum_{i,j=1}^N x_{ij}^2 \le N^2\, ,
\]
which implies the important observation that $\lambda_i^2\le N^2/i$ for each $i$ since $|\lambda_1|\ge|\lambda_2|\ge \cdots \ge |\lambda_n|$. As a result of this, if $M'$ is the matrix obtained from $M(x)$ after throwing away the terms corresponding the $\lambda_{i+1},\ldots,\lambda_n$ in its spectral decomposition, then 
\[
\|M(x)-M'\|_{\text{op}}\le \frac{N}{\sqrt{i+1}}\, .
\]
In other words, $M(x)$ may be approximated by a rank $i$ matrix if we allow $O(Ni^{-1/2})$ error of approximation in the operator norm. But we need only $O(Ni \log N)$ bits to encode a rank $i$ matrix. Taking $i=\ep^{-4}$, and combining with the inequality \eqref{mainmatrix}, it is now easy to see how the quantities $a_{ij}(x)$ may be encoded by $O(N \ep^{-4}\log N)$ bits with $O(\ep)$ error in approximation for a typical $a_{ij}$, on average. This proves the low complexity gradient condition for triangle counts. The proof for general subgraph counts is a messy but straightforward generalization of the above argument.

\section{Proof of Theorem \ref{freethm}}\label{expo}
In this section, we fill out the gaps in the sketch given in the previous section and thereby produce a complete proof of Theorem \ref{freethm}.

Throughout this section, we will freely use the notation of Theorem \ref{freethm}. In particular, $F$, $f$, $f_i$, $f_{ij}$, $a$, $b_i$, $c_{ij}$ and $\dd(\ep)$ are as in the statement of Theorem~\ref{freethm}. Let us also define some additional notation, as follows. (Some of this has already been introduced in the previous section, but we will repeat the definitions here just in case the reader has skipped that part.)

 We will generally denote the $i$th coordinate of a vector $x\in \rr^n$ by $x_i$. Similarly, the $i$th coordinate of a random vector $X$ will be denoted by $X_i$. Given $x\in [0,1]^n$, define $x^{(i)}$ to be the vector $(x_1,\ldots, x_{i-1}, 0, x_{i+1},\ldots, x_n)$. For a random vector $X$ define $X^{(i)}$ similarly. Given a function $g:[0,1]^n \ra\rr$, define the discrete derivative $\Delta_i g$ as 
\begin{align*}
\Delta_i g(x) &:= g(x_1,\ldots, x_{i-1}, 1,x_{i+1},\ldots, x_n)- g(x_1,\ldots, x_{i-1}, 0,x_{i+1},\ldots, x_n).
\end{align*}
For each $i$, define a function $\hx_i:[0,1]^n \ra [0,1]$ as
\begin{align*}
\hx_i(x) = \frac{1}{1+e^{-\Delta_i f(x)}}. 
\end{align*}
Let $\hx:[0,1]^n \ra[0,1]^n$ be the vector-valued function whose $i$th coordinate function is $\hx_i$. When the vector $x$ is understood from the context, we will simply write $\hx$ and $\hx_i$ instead of $\hx(x)$ and $\hx_i(x)$. The proof of Theorem \ref{freethm} requires two key lemmas. %The first lemma shows that $f(X)\approx f(\HX)$ with high probability.
\begin{lmm}\label{meanlmm1}
Let $X=(X_1,\ldots, X_n)$ be a random vector that has probability density proportional to $e^{f(x)}$ on $\{0,1\}^n$ with respect to the counting measure.  Let $\HX = \hx(X)$. Then
\[
\ee\bigl[(f(X) - f(\HX))^2 \bigr]\le \sum_{i=1}^n(ac_{ii} + b_i^2)+ \frac{1}{4}\sum_{i,j=1}^n \bigl(ac_{ij}^2 +  b_i b_j c_{ij}\bigr).
\]
\end{lmm}
\begin{proof}
It is easy to see that
\[
\hx_i(x) = \ee(X_i\mid X_j = x_j, \, 1\le j\le n, \ j\ne i). 
\]
Let $D := f(X)-f(\HX)$.
Then clearly
\begin{align}\label{dest}
|D|\le 2a. 
\end{align}
Define
\[
h(x) := f(x) - f(\hx(x)),
\]
so that $D = h(X)$. Note that  for $i\ne j$, 
\begin{align*}
\fpar{\hx_j}{x_i}  &= \frac{e^{-\Delta_j f(x)}}{(1+e^{-\Delta_j f(x)})^2} \int_0^1 f_{ij}(x_1,\ldots, x_{j-1}, t, x_{j+1},\ldots, x_n)\, dt\, ,
\end{align*}
and for $i=j$, the above derivative is identically equal to zero. 
Since $e^{-x}/(1+e^{-x})^2\le 1/4$ for all $x\in \rr$, this shows that for all $i$ and $j$, 
\begin{equation}\label{gest}
\biggl\|\fpar{\hx_j}{x_i}\biggr\| \le \frac{c_{ij}}{4}.
\end{equation}
Thus,  
\begin{align}
\biggl\|\fpar{h}{x_i}\biggr\| &\le \|f_i\| + \sum_{j=1}^n \|f_j\| \biggl\|\fpar{\hx_j}{x_i}\biggr\|\label{hest}\\
&\le b_i + \frac{1}{4}\sum_{j=1}^n b_j c_{ij}. \nonumber 
\end{align}
Consequently, if $D_i := h(X^{(i)})$, 
then 
\begin{align}\label{dbd}
|D-D_i| \le b_i + \frac{1}{4}\sum_{j=1}^n b_j c_{ij}. 
\end{align}
For $t\in [0,1]$ and $x\in [0,1]^n$ define
\begin{align*}
u_i(t, x) &:= f_i(tx +(1-t)\hx), 
\end{align*}
so that
\[
h(x) = \int_0^1 \sum_{i=1}^n (x_i-\hx_i) u_i(t,x)\, dt. 
\]
Thus,
\begin{align}\label{d2}
\ee(D^2) &= \int_0^1 \sum_{i=1}^n \ee((X_i-\HX_i) u_i(t, X) D) \, dt. 
\end{align}
Now,  
\begin{align}\label{uiest}
\|u_i\| \le  b_i,
\end{align}
and  by \eqref{gest}, 
\begin{align}
\biggl\|\fpar{u_i}{x_i}\biggr\| &\le t \|f_{ii}\|  + (1-t)\sum_{j=1}^n \|f_{ij}\|\biggl\|\fpar{\hx_j}{x_i}\biggr\|\label{uiiest}\\
&\le tc_{ii}+ \frac{1-t}{4}\sum_{j=1}^n c_{ij}^2. \nonumber
\end{align}
The bounds \eqref{dest}, \eqref{dbd}, \eqref{uiest} and \eqref{uiiest} imply that 
\begin{align*}
&\bigl|\ee((X_i-\HX_i) u_i(t, X) D)  - \ee((X_i-\HX_i) u_i(t, X^{(i)}) D_i)\bigr|\\
&\le \ee\bigl|\bigl(u_i(t, X)  - u_i(t, X^{(i)})\bigr)D\bigr|  + \ee\bigl|u_i(t, X^{(i)}) (D-D_i)\bigr|\\
&\le 2atc_{ii}+ \frac{a(1-t)}{2}\sum_{j=1}^n c_{ij}^2 + b_i^2 +  \frac{1}{4}\sum_{j=1}^n b_i b_j c_{ij}. 
\end{align*}
But $ u_i(t, X^{(i)}) D_i$ is a function of the random variables $(X_j)_{j\ne i}$ only. Therefore by the definition of $\HX_i$, 
\[
\ee((X_i-\HX_i) u_i(t, X^{(i)}) D_i) = 0. 
\]
Thus,
\begin{align}
\bigl|\ee((X_i-\HX_i) u_i(t, X) D)\bigr| &\le 2atc_{ii}+ \frac{a(1-t)}{2}\sum_{j=1}^n c_{ij}^2 + b_i^2 +  \frac{1}{4}\sum_{j=1}^n b_i b_j c_{ij}.\nonumber
\end{align}
Using this bound in \eqref{d2} gives 
\begin{align*}
\ee(D^2)&\le \int_0^1 \sum_{i=1}^n\biggl(2atc_{ii}+ \frac{a(1-t)}{2}\sum_{j=1}^n c_{ij}^2 + b_i^2 +  \frac{1}{4}\sum_{j=1}^n b_i b_j c_{ij}\biggr)\, dt\\
&= \sum_{i=1}^n(ac_{ii} + b_i^2)+ \frac{1}{4}\sum_{i,j=1}^n \bigl(ac_{ij}^2 +  b_i b_j c_{ij}\bigr),
\end{align*}
completing the proof.
\end{proof}
\begin{lmm}\label{meanlmm2}
Let all notation be as in Lemma \ref{meanlmm1}. Then
\[
\ee\biggl[\biggl( \sum_{i=1}^n (X_i - \HX_i) \Delta_i f(X)\biggr)^2\biggr]\le \sum_{i=1}^n b^2_i + \frac{1}{4}\sum_{i,j=1}^n b_i(b_j+4) c_{ij}. 
\]
\end{lmm}
\begin{proof}
Let $g_i$ denote the function $\Delta_i f$, for notational simplicity. Note that 
\[
g_i(x) = \int_0^1 f_i(x_1,\ldots, x_{i-1}, t, x_{i+1},\ldots, x_n)\, dt,
\]
which shows that
\begin{equation}\label{g1}
\|g_i\|\le \|f_i\| = b_i
\end{equation}
and for all $j$,
\begin{equation}\label{g2}
\biggl\|\fpar{g_i}{x_j}\biggr\|\le \|f_{ij}\|= c_{ij}.
\end{equation}
Let
\[
G(x) := \sum_{i=1}^n(x_i-\hx_i(x)) g_i(x).
\]
%Then by \eqref{g1}, 
%\[
%\|G\|\le \sum_{i=1}^n \|g_i\| \le \sum_{i=1}^n b_i. 
%\] 
%Next, note that 
Then 
\begin{align*}
\fpar{G}{x_i} &= \sum_{j=1}^n \biggl[\biggl(1_{\{j=i\}} - \fpar{\hx_j}{x_i}\biggr)g_j(x) + (x_j-\hx_j) \fpar{g_j}{x_i}\biggr]
\end{align*}
and therefore by \eqref{gest}, \eqref{g1} and \eqref{g2},
\begin{align}\label{g3}
\biggl\|\fpar{G}{x_i}\biggr\| &\le b_i + \frac{1}{4}\sum_{j=1}^n c_{ij}b_j + \sum_{j=1}^n c_{ij}. 
\end{align}
%Next, we write 
%\[
%\ee(G^2) = \sum_{i=1}^n \ee((X_i - \HX_i) g_i(X) G(X)). 
%\]
Note that for any $x$, 
\begin{align}\label{g4}
|G(x) - G(x^{(i)})| &\le \biggl\|\fpar{G}{x_i}\biggr\|.
\end{align}
Again, $g_i(X)$ and $G(X^{(i)})$ are both functions of $(X_j)_{j\ne i}$ only. Therefore
\begin{align}\label{g5}
\ee((X_i - \HX_i) g_i(X) G(X^{(i)})) = 0. 
\end{align}
Combining \eqref{g3}, \eqref{g4} and \eqref{g5} gives 
\begin{align*}
\ee(G(X)^2) &= \sum_{i=1}^n\ee((X_i-\HX_i) g_i(X) G(X))\\
&\le \sum_{i=1}^n b_i\biggl( b_i + \frac{1}{4}\sum_{j=1}^n c_{ij}b_j + \sum_{j=1}^n c_{ij}\biggr).
\end{align*}
This completes the proof of the lemma.
\end{proof}
With the aid of Lemma \ref{meanlmm1} and \ref{meanlmm2}, we are now ready to prove Theorem~\ref{freethm}. 
\begin{proof}[Proof of the upper bound in Theorem \ref{freethm}]
For $x,y\in [0,1]^n$, let
\[
g(x, y) := \sum_{i=1}^n(x_i \log y_i + (1-x_i)\log (1-y_i)). 
\]
Note that 
\begin{align}\label{gidiff}
g(x, \hx)-I(\hx)  &= \sum_{i=1}^n (x_i - \hx_i) \log \frac{\hx_i}{1-\hx_i} = \sum_{i=1}^n (x_i - \hx_i) \Delta_i f(x).
\end{align}
Let
\[
B := 4\biggl(\sum_{i=1}^n(ac_{ii} + b_i^2)+ \frac{1}{4}\sum_{i,j=1}^n \bigl(ac_{ij}^2+b_i b_j c_{ij} + 4b_i c_{ij}\bigr)\biggr)^{1/2}.
\]
%and
%\[
%\ep_2 := 2\biggl( \sum_{i=1}^n b^2_i + \frac{1}{4}\sum_{i,j=1}^n b_i(b_j+1) c_{ij}\biggr)^{1/2}
%\]
Let 
\[
A_1 := \{x\in \{0,1\}^n : |I(\hx)-g(x,\hx)|\le B/2\}, 
\]
and 
\[
A_2 := \bigl\{x\in \{0,1\}^n: |f(x)-f(\hx)| \le B/2\bigr\}.
\]
Let $A=A_1 \cap A_2$.  By Lemma \ref{meanlmm2} and the identity \eqref{gidiff}, $\pp(X\not \in A_1)\le 1/4$. By Lemma \ref{meanlmm1}, $\pp(X\not \in A_2)\le 1/4$. Thus, 
\[
\pp(X\in A) \ge \frac{1}{2}. 
\]
That is, 
\[
\frac{\sum_{x\in A} e^{f(x)}}{\sum_{x\in \{0,1\}^n} e^{f(x)}} \ge \frac{1}{2},
\]
and therefore by the definition of the set $A$,
\begin{align}
F &= \log \sum_{x\in \{0,1\}^n} e^{f(x)}\le \log \sum_{x\in A} e^{f(x)} + \log 2 \label{fineq1}\\
&\le B + \log \sum_{x\in A} e^{f(\hx) -I(\hx)+g(x, \hx)} + \log 2.\nonumber
\end{align}
Now take some $x\in [0,1]^n$ and let $d$ satisfy \eqref{entropy}. Then by the Cauchy-Schwarz inequality,  
\[
\sum_{i=1}^n |f_i(x)- d_i|\le n \ep. 
\]
Fix such an $x$ and $d$. Note that for each $i$,
\begin{align*}
|\Delta_i f(x) - f_i(x)| &\le \int_0^1 |f_i(x_1,\ldots, x_{i-1}, t, x_{i+1},\ldots, x_n) - f_i(x)| \, dt\\
&\le \|f_{ii}\| = c_{ii}. 
\end{align*}
By the last two inequalities and \eqref{entropy}, 
\begin{equation}\label{fd1}
\sum_{i=1}^n |\Delta_i f(x)- d_i|\le n \ep + \sum_{i=1}^n c_{ii}. 
\end{equation}
and 
\begin{equation}\label{fd2}
\biggl(\sum_{i=1}^n (\Delta_i f(x)- d_i)^2\biggr)^{1/2}\le n^{1/2} \ep + \Big(\sum_{i=1}^n c_{ii}^2\Big)^{1/2}. 
\end{equation}
Let $u(x) = 1/(1+e^{-x})$. Note that for all $x$, 
\[
|u'(x)| = \frac{1}{(e^{x/2} + e^{-x/2})^2}\le \frac{1}{4}. 
\]
Therefore if a vector $p = p(d)$ is defined as $p_i = u(d_i)$, then  by \eqref{fd2},
\begin{align*}
\biggl(\sum_{i=1}^n(\hx_i - p_i)^2\biggr)^{1/2}&\le \biggl(\frac{1}{16} \sum_{i=1}^n (\Delta_i f(x)- d_i)^2\biggr)^{1/2} \\
&\le \frac{n^{1/2}\ep}{4}+\frac{1}{4}\Big(\sum_{i=1}^n c_{ii}^2\Big)^{1/2}. 
\end{align*}
Thus, if 
\[
L := \Big(\sum_{i=1}^n b_i^2\Big)^{1/2}, 
\]
then
\begin{align}
|f(\hx)- f(p)|&\le L \biggl(\sum_{i=1}^n(\hx_i - p_i)^2\biggr)^{1/2} \label{pd2}\\
&\le \frac{Ln^{1/2}\ep}{4}+\frac{L}{4}\Big(\sum_{i=1}^n c_{ii}^2\Big)^{1/2}.\nonumber
\end{align}
Next, let $v(x)=\log(1+e^{-x})$. Then for all $x$, 
\[
|v'(x)|= \frac{e^{-x}}{1+e^{-x}} \le 1.
\]
Consequently, 
\[
|\log \hx_i - \log p_i|\le |\Delta_i f(x) - d_i|
\]
and 
\[
|\log(1- \hx_i) - \log (1-p_i)|\le |\Delta_i f(x) - d_i|.
\]
Therefore by \eqref{fd1},
\begin{align}\label{pd1}
| g(x, \hx) -g(x,p)|&\le 2\sum_{i=1}^n |\Delta_i f(x)-d_i| \le 2n\ep+2\sum_{i=1}^n c_{ii}. 
\end{align}
Finally, let $w(x)=I(u(x))$. Then
\begin{align*}
w'(x) &= u'(x)I'(u(x)) \\
&= \frac{e^{-x}}{(1+e^{-x})^2}\log \frac{u(x)}{1-u(x)}\\
&= \frac{xe^{-x}}{(1+e^{-x})^2}. 
\end{align*}
Thus, for all $x$,
\[
|w'(x)|\le \sup_{x\in \rr} \frac{|x|e^{-x}}{(1+e^{-x})^2}\le \sup_{x\ge 0} xe^{-x} = \frac{1}{e}.
\]
Thus,
\[
|I(\hx_i)-I(p_i)|\le \frac{1}{e}|\Delta_i f(x)- d_i|, 
\]
and so by \eqref{fd1}, 
\begin{equation}\label{pd3}
|I(\hx)- I(p)|\le \frac{n\ep}{e} + \frac{1}{e}\sum_{i=1}^n c_{ii}. 
\end{equation}
For each $d\in \dd(\ep)$ let $\mc(d)$ be the set of all $x\in \{0,1\}^n$ such that \eqref{entropy} holds, and let $p(d)$ be the vector $p$ defined above. Then by \eqref{pd2}, \eqref{pd1} and \eqref{pd3}, 
\begin{align}
&\log \sum_{x\in A} e^{f(\hx)-I(\hx) + g(x, \hx )} \le \log \sum_{d\in \dd(\ep)}\sum_{x\in \mc(d)} e^{f(\hx)-I(\hx) + g(x, \hx)} \label{fineq2} \\
&\le \frac{Ln^{1/2}\ep}{4}+\frac{L}{4}\Big(\sum_{i=1}^n c_{ii}^2\Big)^{1/2} + 2n\ep+2\sum_{i=1}^n c_{ii} + \frac{n\ep}{e} + \frac{1}{e}\sum_{i=1}^n c_{ii} \nonumber \\
&\qquad  + \log \sum_{d\in \dd(\ep)}\sum_{x\in \mc(d)} e^{f(p(d))-I(p(d)) + g(x,p(d))}. \nonumber
\end{align}
Now note that for any $p\in [0,1]^n$, 
\[
\sum_{x\in \{0,1\}^n} e^{g(x,p)} = 1. 
\]
Thus,
\begin{align}
 \log \sum_{d\in \dd(\ep)}&\sum_{x\in \mc(d)} e^{f(p(d))-I(p(d)) + g(x,p(d))} \label{fineq3}\\
 &\le  \log \sum_{d\in \dd(\ep)}e^{f(p(d))-I(p(d))}\nonumber\\
 &\le \log|\dd(\ep)| + \sup_{p\in [0,1]^n} (f(p)-I(p)). \nonumber
\end{align}
Combining \eqref{fineq1}, \eqref{fineq2} and \eqref{fineq3}, the proof is complete.
\end{proof}
\begin{proof}[Proof of the lower bound in Theorem \ref{freethm}]
Fix some $y\in [0,1]^n$. Let $Y = (Y_1,\ldots, Y_n)$ be a random vector with independent components, where $Y_i$ is a $Bernoulli(y_i)$ random variable. Then by Jensen's inequality,
\begin{align*}
\sum_{x\in \{0,1\}^n} e^{f(x)} &= \sum_{x\in \{0,1\}^n} e^{f(x)-g(x,y)+g(x,y)}\\
&= \ee(e^{f(Y)-g(Y,y)})\\
&\ge \exp(\ee(f(Y) - g(Y,y)))\\
&= \exp(\ee(f(Y)) - I(y))\, .
\end{align*}
Let $S := f(Y)- f(y)$. 
For $t\in [0,1]$ and $x\in [0,1]^n$ define
\begin{align*}
v_i(t, x) &:= f_i(tx +(1-t)y)\, , 
\end{align*}
so that
\begin{equation}\label{d2new}
S = \int_0^1 \sum_{i=1}^n (Y_i-y_i) v_i(t,Y)\, dt\, . 
\end{equation}
By the independence of $Y_i$ and $Y^{(i)}$,  
\begin{align*}
\bigl|\ee((Y_i-y_i) v_i(t, Y) )\bigr| &= \bigl|\ee((Y_i-y_i)( v_i(t, Y)- v_i(t, Y^{(i)})))\bigr|\\
&\le \biggl\|\fpar{v_i}{x_i}\biggr\|\le t c_{ii}\, . 
\end{align*}
Using this bound in \eqref{d2new} gives 
\begin{align*}
\ee(S)&\ge -\int_0^1 \sum_{i=1}^n tc_{ii} \, dt= -\frac{1}{2}\sum_{i=1}^nc_{ii}\, .
\end{align*}
This completes the proof.
\end{proof}

\section{Proof of Theorem \ref{uppertailthm}}\label{large}
Throughout this section, we will use the notation of Theorem \ref{uppertailthm} without explicit mention.  
\begin{proof}[Proof of the upper bound in Theorem \ref{uppertailthm}]
Let $h:\rr \ra\rr$ be a function that is twice continuously differentiable, non-decreasing,  and satisfies $h(x)=-1$ if $x\le -1$ and $h(x)=0$ if $x\ge 0$. Let $L_1:=\|h'\|$ and $L_2:= \|h''\|$. A specific choice of $h$ is given by $h(x) = 10(x+1)^3 -15(x+1)^4 + 6(x+1)^5-1$ for $-1\le x\le 0$, which gives $L_1 \le 2$ and $L_2 \le 6$. 
Define
\[
\psi(x) := K h((x-t)/\delta). 
\]
Then clearly
\[
\|\psi\|\le K, \hskip.2in \|\psi'\|\le \frac{L_1K}{\delta}, \hskip.2in \|\psi''\|\le \frac{L_2K}{\delta^2}\, . 
\]
Let 
\[
g(x) := n\psi(f(x)/n) + \sum_{i=1}^n (x_i \log p + (1-x_i) \log(1-p))\, .
\]
The plan is to apply Theorem \ref{freethm} to the function $g$ instead of $f$. Note that $\psi(x) = 0$ if $x\ge t$. Thus, 
\begin{align*}
\pp(f(Y) \ge tn) &\le \ee(e^{n\psi(f(Y)/n)}) \\
&= \sum_{x\in \{0,1\}^n} e^{g(x)}\, .
\end{align*}
Note also that for any $x\in [0,1]^n$ such that $f(x)\ge tn$, 
\[
g(x)-I(x) = n \psi(f(x)/n) - I_p(x) = -I_p(x) \le -\phi_p(t).
\] 
Again, if $f(x)\le (t-\delta)n$, then $(f(x)/n-t)/\delta \le -1$, and so
\[
g(x)-I(x) = -nK-I_p(x)\le -nK = -\phi_p(t)\, . 
\]
Finally, note that if $f(x) = (t-\delta')n$ for some $0< \delta' < \delta$, then 
\[
g(x)-I(x) \le -I_p(x) \le -\phi_p(t-\delta')\le -\phi(t-\delta)\, . 
\]
Thus,
\begin{equation*}\label{supgi}
\sup_x(g(x)-I(x)) \le -\phi_p(t-\delta)\, .
\end{equation*}
Let $C_p := |\log p| + |\log (1-p)|$. Note that 
\[
\|g\|\le nK + nC_p = \alpha\, ,
\]
and for any $i$,
\begin{align*}
\biggl\|\fpar{g}{x_i}\biggr\| &\le \frac{2K b_i}{\delta} + C_p = \beta_i\, ,
\end{align*}
and for any $i,j$, 
\begin{align*}
\biggl\|\mpar{g}{x_i}{x_j}\biggr\| &\le \frac{2K c_{ij}}{\delta} +\frac{ 6K b_i b_j}{n\delta^2} = \gamma_{ij}\, .
\end{align*}
Next, fix some $\ep > 0$ and let $\dd(\ep)$ be as in Section \ref{expo}. Let
\[
\ep' := \frac{\ep}{2\|\psi'\|}\, , \ \ \ \tau := \frac{\ep}{2\big(\frac{1}{n}\sum_{i=1}^n b_i^2\big)^{1/2}}\, .
\]
Let $l\in \rr^n$ be the vector whose coordinates are all equal to $\log(p/(1-p))$ and define
\[
\dd'(\ep) := \{\theta d + l : d\in \dd(\ep'), \, \theta = j\tau \text{ for some integer } 0\le j< \|\psi'\|/\tau\}\, .
\]
Let $g_i := \partial g/\partial x_i$. Take any $x\in [0,1]^n$, and choose $d\in \dd(\ep)$ satisfying \eqref{entropy}. Choose an integer $j$ between $0$ and $\|\psi'\|/\tau$ such that $|\psi'(f(x)/n) - j\tau|\le \tau$. Let $d' := j\tau d+ l$, so that $d'\in \dd'(\ep)$. Then
\begin{align*}
&\sum_{i=1}^n (g_i(x) - d'_i)^2 = \sum_{i=1}^n (\psi'(f(x)/n)f_i(x) - j\tau d_i)^2\\
&\le 2(\psi'(f(x)/n)- j\tau)^2 \sum_{i=1}^n f_i(x)^2 + 2\|\psi'\|^2 \sum_{i=1}^n (f_i(x)-d_i)^2\\
&\le  2\tau^2 \sum_{i=1}^n b_i^2 + 2\|\psi'\|^2 n \ep'^2 = n\ep^2. 
\end{align*}
This shows that $\dd'(\ep)$ plays the role of $\dd(\ep)$ for the function $g$. Note that
\[
|\dd'(\ep)|\le \frac{\|\psi'\|}{\tau} |\dd(\ep')|\, .
\]
This gives the upper bound on the complexity term for the function $g$. The proof is completed by applying Theorem~\ref{freethm}. 
\end{proof}

\begin{proof}[Proof of the lower bound in Theorem \ref{uppertailthm}]
Fix any $z\in [0,1]^n$ such that 
\[
f(z)\ge (t+\delta_0) n\, .
\]
Let $Z = (Z_1,\ldots, Z_n)$ be a random vector with independent components, where $Z_i\sim Bernoulli(z_i)$.  Let $\ma$ be the set of all $x\in \{0,1\}^n$ such that $f(x)\ge tn$. Let $\ma'$ be the subset of $\ma$ where $|g(x,z)-g(x,p) - I_p(z)|\le \ep_0 n$. Then
\begin{align}
\pp(f(Y)\ge tn) &= \sum_{x\in \ma} e^{g(x,p)}\label{peq0}\\
&= \sum_{x\in \ma} e^{g(x,p)-g(x,z) + g(x,z)} \nonumber \\
&\ge \sum_{x\in \ma'} e^{g(x,p)-g(x,z) + g(x,z)}\ge e^{-I_p(z)-\ep_0 n} \pp(Z\in \ma')\, .\nonumber 
\end{align}
Note that 
\[
\ee(g(Z, z) - g(Z, p)) = I_p(z)\, ,
\]
and 
\begin{align*}
&\var(g(Z, z) - g(Z, p)) \\
&= \sum_{i=1}^n \var(Z_i\log (z_i/p) + (1-Z_i)\log((1-z_i)/(1-p))) \\
&= \sum_{i=1}^n z_i(1-z_i) \biggl(\log \frac{z_i/p}{(1-z_i)/(1-p)}\biggr)^2\, . 
\end{align*}
Using the inequalities $|\sqrt{x}\log x|\le 2/e\le 1$ and $x(1-x)\le 1/4$, we see that for any $x\in [0,1]$, 
\begin{align*}
&x(1-x) \biggl(\log \frac{x/p}{(1-x)/(1-p)}\biggr)^2\\
&\le \biggl(|\sqrt{x}\log x| + |\sqrt{1-x}\log (1-x)| +\frac{1}{2}\biggl|\log \frac{p}{1-p}\biggr|\biggr)^2\\
&\le \biggl(2 +\frac{1}{2}\biggl|\log \frac{p}{1-p}\biggr|\biggr)^2\, .
\end{align*}
Combining the last three displays, we see that
\begin{align}\label{peq1}
\pp(|g(Z,z)-g(Z, p) - I_p(z)| > \ep_0 n) \le \frac{1}{\ep_0^2 n}\biggl(2 +\frac{1}{2}\biggl|\log \frac{p}{1-p}\biggr|\biggr)^2 = \frac{1}{4}\, .
\end{align}
Let $S := f(Z)-f(z)$ and $v_i(t,x) := f_i(tZ + (1-t) z)$. Let $S_i := f(Z^{(i)}) - f(z)$, so that $|S-S_i|\le b_i$. Since
\[
S = \int_0^1  \sum_{i=1}^n (Z_i-z_i) v_i(t,Z)\, dt\, ,
\]
we have
\begin{align}\label{d23}
\ee(S^2) &= \int_0^1  \sum_{i=1}^n \ee((Z_i-z_i) v_i(t,Z)S)\, dt\, .
\end{align}
By the independence of $Z_i$ and the pair $(S_i, Z^{(i)})$, 
\begin{align*}
&\bigl|\ee((Z_i-z_i) v_i(t,Z)S)\bigr| \\
&= \bigl|\ee((Z_i-z_i) (v_i(t,Z)S - v_i(t, Z^{(i)})S_i))\bigr|\\
&\le \|S\|\biggl\|\fpar{v_i}{x_i}\biggr\|+ \|v_i\| \|S-S_i\|\\
&\le 2at c_{ii}+ b_i^2\, .
\end{align*}
By \eqref{d23}, this gives 
\[
\ee(S^2) \le \sum_{i=1}^n (ac_{ii}+b_i^2)\, .
\]
Therefore,
\begin{align}\label{peq2}
\pp(f(Z)< tn) \le \frac{1}{\delta_0^2n^2}\sum_{i=1}^n (ac_{ii}+b_i^2) = \frac{1}{4}\, .
\end{align}
Inequalities \eqref{peq1} and \eqref{peq2} give
\[
\pp(Z\in \ma') \ge \frac{1}{2}\, .
\]
Plugging this into \eqref{peq0} and taking supremum over $z$ completes the proof. 
\end{proof}

\section{Proof of Theorem \ref{graphthm}}\label{graphsec}
Let all notation be the same as in the statement of Theorem \ref{graphthm}. Let 
\[
n := {N \choose  2}.
\]
Throughout this section, we will index the elements of $\rr^n$ as 
\[
x = (x_{ij})_{1\le i<j\le N}\, ,
\]
with the understanding that if $i<j$, then $x_{ji}$ is the same as $x_{ij}$, and for all $i$, $x_{ii}=0$. Let $k$ be a positive integer, and let $H$ be a finite simple graph on the vertex set $[k] := \{1,\ldots, k\}$. Let $E$ be the set of edges of $H$ and let $m:= |E|$. 

%Let $\psi:[0,1]\ra\rr$ be a continuous function that is $C^2$ in $(0,1)$. 
Define a function $T:[0,1]^n \ra\rr$ as 
\begin{align}\label{tdef}
T(x) &:= \frac{1}{N^{k-2}}\sum_{q\in [N]^k} \prod_{\{l,l'\}\in E}x_{q_lq_{l'}}\, ,
\end{align}
so that $t(H,G_x) = T(x)/N^2$. 
%For example, if $H$ is a triangle, then 
%\[
%T(x) = \frac{1}{N^3}\sum_{1\le q_1, q_2, q_3\le N} x_{q_1q_2}x_{q_2q_3}x_{q_3q_1}. 
%\]
The plan is to apply Theorem \ref{uppertailthm} with $f=T$. We will now compute the required bounds for the function $T$.
\begin{lmm}\label{test1}
For the function $T$ on $\rr^n$ defined above, $\|T\|\le N^2$, and for any $i<j$ and $i'<j'$, 
\begin{align*}
&\biggl\|\fpar{T}{x_{ij}}\biggr\| \le 2m\, , \text{ and}\\
&\biggl\|\mpar{T}{x_{ij}}{x_{i'j'}}\biggr\|\le 
\begin{cases}
4m(m-1) N^{-1} \ &\text{ if $|\{i,j,i',j'\}| = 2$ or $3$\, ,} \\
4m(m-1) N^{-2} & \text{ if $|\{i,j,i',j'\}| = 4$\, .}
\end{cases} 
\end{align*}
\end{lmm}
\begin{proof}
It is clear that $\|T\|\le N^2$ since the $x_{ij}$'s are all in $[0,1]$ and there are exactly $N^k$ terms in the sum that defines $T$. Next, note that for any $i<j$, 
\begin{align}\label{tderiv}
\fpar{T}{x_{ij}} &= \frac{1}{N^{k-2}}\sum_{\{a,b\}\in E}\sum_{\substack{q\in [N]^k\\ \{q_a, q_b\} = \{i,j\}}} \prod_{\substack{\{l,l'\}\in E\\ \{l,l'\}\ne \{a, b\}}}x_{q_lq_{l'}}\, , % + \log \frac{p}{1-p}\ , 
\end{align}
and therefore
\begin{equation*}%\label{ttbd}
\biggl\|\fpar{T}{x_{ij}}\biggr\| \le \frac{2m N^{k-2}}{N^{k-2}} =2m\, . % + \biggl|\log \frac{p}{1-p}\biggr|. 
\end{equation*}
Next, for any $i<j$ and $i'<j'$, 
\begin{align*}
\mpar{T}{x_{ij}}{x_{i'j'}} &= \frac{1}{N^{k-2}}\sum_{\{a,b\}\in E} \sum_{\substack{\{c,d\}\in E\\ \{c,d\}\ne \{a,b\}}}\sum_{\substack{q\in [N]^k\\ \{q_a, q_b\} = \{i, j\}\\ \{q_c, q_d\} = \{i', j'\}}} \prod_{\substack{\{l,l'\}\in E\\ \{l,l'\}\ne \{a,b\}\\ \{l,l'\}\ne \{c, d\}}}x_{q_lq_{l'}}\ . 
\end{align*}
Take any two edges $\{a,b\}, \{c,d\}\in E$ such that $\{a,b\}\ne \{c,d\}$. Then the number of choices of $q\in [N]^k$ such that $\{q_a,  q_b\} = \{i,j\}$ and $\{q_c ,q_d\} = \{i',j'\}$ is at most $4N^{k-3}$ if $|\{i,j,i',j'\}|=2$ or $3$ (since we are constraining $q_a$, $q_b$, $q_c$ and $q_d$ and $|\{a,b,c,d\}|\ge 3$ always), and at most $4N^{k-4}$ if $|\{i,j,i',j'\}|=4$ (since $|\{a,b,c,d\}|$ must be $4$ if there is at least one possible choice of $q$ for these $i,j,i',j'$). 
This gives the upper bound for the second derivatives. 
\end{proof}

\begin{lmm}\label{test2}
For the function $T$ defined above, one can produce sets $\dd(\ep)$ satisfying the criterion \eqref{entropy} (with $f=T$) such that 
\[
|\dd(\ep)|\le  \exp\biggl(\frac{C_1 m^4 k^4N}{\ep^4}\log \frac{C_2m^4 k^4}{\ep^4}\biggr)\, ,
\]
where $C_1$ and $C_2$ are universal constants. 
\end{lmm}
The proof of Lemma \ref{test2} requires some preparation. We begin by introducing some special notation. For an $N\times N$ matrix $M$, recall the definition of the operator norm:
\[
\|M\|_{\textup{op}} :=\max \{\|Mx\|: x\in \rr^N\, , \, \|x\|=1\}\, . %\max\{|\lambda|: \text{$\lambda$ is an eigenvalue of $M$}\}\, . 
\]
%Alternatively, $\|M\|_{\textup{op}}$ is the largest singular value of $M$. 
For $x = (x_{ij})_{1\le i<j\le N}\in \rr^n$, let $M(x)$ be the symmetric matrix whose $(i,j)$th entry is $x_{ij}$, with the convention that $x_{ij}=x_{ji}$ and $x_{ii}=0$.  Define the operator norm on $\rr^n$ as 
\[
\|x\|_{\textup{op}} := \|M(x)\|_{\textup{op}}.
\]
The following lemma estimates the entropy of the unit cube under this norm.
\begin{lmm}\label{exupper21}
For any $\tau\in (0,1)$, there is a finite set of $N\times N$ matrices $\mw(\tau)$ such that 
\[
|\mw(\tau)|\le e^{34(N/\tau^{2}) \log(51/\tau^{2})}\, ,
\]
and for any $N\times N$ matrix $M$ with entries in $[0,1]$, there exists $W\in \mw(\tau)$ such that  
\[
\|M-W\|_{\textup{op}}\le N\tau\, . 
\]
In particular, for any $x\in [0,1]^n$ there exists $W\in \mw(\tau)$ such that $\|M(x)-W\|_{\textup{op}}\le N\tau$. 
\end{lmm}
\begin{proof}
Let $l$ be the integer part of $17/\tau^2$ and $\delta = 1/l$. Let $\ma$ be a finite subset of the unit ball of $\rr^N$ such that any vector inside the ball is at Euclidean distance $\le \delta$ from some element of $\ma$. (In other words, $\ma$ is a $\delta$-net of the unit ball under the Euclidean metric.) The set $\ma$ may be defined as a maximal set of points in the unit ball such that any two are at a distance greater than $\delta$ from each other. Since the balls of radius $\delta/2$ around these points are disjoint and their union is contained in the ball of radius $1+\delta/2$ centered at zero, it follows that $|\ma|C(\delta/2)^N\le C(1+\delta/2)^N$, where $C$ is the volume of the unit ball. Therefore,
\begin{align}\label{asize}
|\ma| \le (3/\delta)^N. 
\end{align}
Take any $x\in \rr^n$. Suppose that $M$ has singular value decomposition
\[
M = \sum_{i=1}^n \lambda_i u_iv_i^t\, , 
\]
where $\lambda_1\ge \lambda_2\ge \cdots \lambda_n \ge 0$ are the singular values of $M$, and $u_1,\ldots, u_n$ and $v_1,\ldots, v_n$ are singular vectors, and $v_i^t$ denotes the transpose of the column vector $v_i$. Assume that the $u_i$'s and $v_i$'s are orthonormal systems.  Since the elements of $M$ all belong to the interval $[0,1]$, it is easy to see that $\lambda_1\le N$ and $\sum \lambda_i^2 \le N^2$. Due to the second inequality, there exists $y\in \ma$  such that 
\begin{equation}\label{lambdadelta}
\sum_{i=1}^N (N^{-1}\lambda_i-y_i)^2 \le \delta^2. 
\end{equation}
Let $z_1,\ldots,z_N$ and $w_1,\ldots, w_N$ be elements of $\ma$ such that for each $i$, 
\begin{equation}\label{uz}
\sum_{j=1}^N(u_{ij} - z_{ij})^2 \le \delta^2\ \text{ and } \ \sum_{j=1}^N(v_{ij} - w_{ij})^2\le \delta^2\, ,
\end{equation}
where $u_{ij}$ denotes the $j$th component of the vector $u_i$, etc. 
Define two matrices $V$ and $W$ as 
\[
V := \sum_{i=1}^{l-1} \lambda_iu_iv_i^t \ \text{ and } \ W := \sum_{i=1}^{l-1} Ny_iz_iw_i^t\, .
\] 
Note that since $\sum \lambda_i^2 \le N^2$ and $\lambda_i$ decreases with $i$, therefore for each $i$, $\lambda_i^2 \le N^2/i$. Thus, 
\begin{align*}
\|M-W\|_{\textup{op}} &\le \|M-V\|_{\textup{op}} + \|V-W\|_{\textup{op}}\\
&\le \frac{N}{\sqrt{l}} + \|V-W\|_{\textup{op}}. 
\end{align*}
Next, note that by \eqref{uz}, the operator norms of the rank-one matrices $(u_i-z_i)v_i^t$ and $z_i(v_i-w_i)^t$ are bounded by $\delta$. And by \eqref{lambdadelta}, $|\lambda_i - Ny_i|\le N\delta$ for each $i$. Therefore
\begin{align*}
\|V-W\|_{\textup{op}} &\le \biggl\|\sum_{i=1}^{l-1} (\lambda_i - Ny_i)u_iv_i^t\biggr\|_{\textup{op}} + \biggl\|\sum_{i=1}^{l-1} Ny_i(u_i-z_i)v_i^t\biggr\|_{\textup{op}}\\
&\qquad + \biggl\|\sum_{i=1}^{l-1} Ny_iz_i(v_i - w_i)^t\biggr\|_{\textup{op}}\\
&\le \max_{1\le i\le l-1} |\lambda_i - Ny_i|  + 2\sum_{i=1}^{l-1} N |y_i| \delta\\
&\le N\delta + 2N\delta \biggl((l-1)\sum_{i=1}^{l-1}y_i^2\biggr)^{1/2}\le N\delta + 2N\delta \sqrt{l-1} \le \frac{3N}{\sqrt{l}}\, . 
\end{align*}
Thus,
\[
\|M-W\|_{\textup{op}}\le \frac{4N}{\sqrt{l}} \le \frac{4N}{\sqrt{\frac{17}{\tau^2}-1}}\le \frac{4N}{\sqrt{\frac{16}{\tau^2}}} = N\tau. 
\]
Let $\mw(\tau)$ be the set of all possible $W$'s constructed in the above manner. Then $\mw(\tau)$ has the required property, and by \eqref{asize},
\begin{align*}
|\mw(\tau)|&\le \text{The number of ways of choosing}\\
&\qquad y, z_1,\ldots, z_{l-1}, w_1,\ldots, w_{l-1}\in \ma\\
&=  |\ma|^{2l-1}\le (3/\delta)^{2Nl} = e^{2Nl\log (3l)}. 
\end{align*}
This completes the proof of the lemma.
\end{proof}
Let $r$ be a positive integer. Let $K_r$ be the complete graph on the vertex set $\{1,\ldots, r\}$. For any set of edges $A$ of $K_r$, any $q = (q_1,\ldots,q_r)\in [N]^r$, and any $x\in [0,1]^n$, let
\[
P(x,q, A) := \prod_{\{a,b\}\in A} x_{q_aq_b}\, , 
\]
with the usual convention that the empty product is $1$. Note that if $q_a=q_b$ for some $\{a,b\}\in A$, the $P(x,q,A)=0$ due to our convention that $x_{ii}=0$ for each $i$. 
Next, note that if $A$ and $B$ are disjoint sets of edges, then 
\begin{equation}\label{pprod}
P(x,q,A\cup B) = P(x, q, A)P(x, q,B). 
\end{equation}
\begin{lmm}\label{comb1}
Let $A$ and $B$ be  sets of edges of $K_r$, and let $e = \{\alpha, \beta\}$ be an edge that is neither in $A$ nor in $B$. Then for any $x,y\in [0,1]^n$, 
\[
\biggl|\sum_{q\in [N]^r}P(x,q,A)P(y,q,B)(x_{q_\alpha q_\beta} - y_{q_\alpha q_\beta})\biggr|\le N^{r-1}\|x-y\|_{\textup{op}}. 
\]
\end{lmm}
\begin{proof}
By relabeling the vertices of $K_r$ and redefining $A$ and $B$, we may assume that $\alpha =1$ and $\beta = 2$. 

Let $A_1$ be the set of all edges in $A$ that are incident to $1$. Let $A_2$ be the set of all edges in $A$ that are incident to $2$. Note that since $\{1,2\}\not\in A$, therefore $A_1$ and $A_2$ must be  disjoint. Similarly, let $B_1$ be the set of all edges in $B$ that are incident to $1$ and let $B_2$ be the set of all edges in $B$ that are incident to $2$. Let $A_3 = A\backslash (A_1\cup A_2)$ and $B_3 = B \backslash (B_1\cup B_2)$. By \eqref{pprod},
\[
P(x,q,A) = P(x,q,A_1)P(x,q,A_2)P(x,q,A_3)
\]
and 
\[
P(y,q,B) = P(y,q,B_1)P(y,q,B_2)P(y,q,B_3). 
\]
Thus,
\begin{align*}
&\sum_{q\in [N]^r}P(x,q,A)P(y,q,B)(x_{q_1 q_2} - y_{q_1 q_2}) \\
&=\sum_{q_3,\ldots, q_r} P(x, q, A_3) P(y,q,B_3) \biggl(\sum_{q_1, q_2} Q(x,y,q)(x_{q_1q_2}-y_{q_1q_2})\biggr),
\end{align*}
where
\[
Q(x,y,q) =  P(x,q,A_1)P(x,q,A_2)P(y,q,B_1)P(y,q,B_2).
\]
Now fix $q_3,\ldots, q_r$. Then $P(x,q,A_1)P(y,q,B_1)$ is a  function of $q_1$ only, and does not depend on $q_2$. Let $g(q_1)$ denote this function. Similarly, $P(x,q,A_2)P(y,q,B_2)$ is a  function of $q_2$ only, and does not depend on $q_1$. Let $h(q_2)$ denote this function. Both $g$ and $h$ are uniformly bounded by $1$. Therefore
\begin{align*}
\biggl|\sum_{q_1, q_2} Q(x,y,q)(x_{q_1q_2}-y_{q_1q_2})\biggr| &= \biggl|\sum_{q_1, q_2} g(q_1)h(q_2)(x_{q_1q_2}-y_{q_1q_2})\biggr|\\
&\le N\|x-y\|_{\textup{op}}. 
\end{align*}
Since this is true for all choices of $q_3,\ldots, q_r$ and $P$ is also uniformly bounded by $1$, this completes the proof of the lemma. 
\end{proof}
Let $A$ and $B$ be two sets of edges of $K_r$. For $x,y\in [0,1]^n$, define
\[
R(x,y,A,B) := \sum_{q\in [N]^r} P(x,q,A)P(y,q,B). 
\]
\begin{lmm}\label{comb2}
Let $A$, $B$, $A'$ and $B'$ be sets of edges of $K_r$ such that $A\cap B = A'\cap B' = \emptyset$ and $A\cup B = A'\cup B'$. Then
\[
|R(x,y,A,B)-R(x,y,A', B')|\le \frac{1}{2}r(r-1)N^{r-1}\|x-y\|_{\textup{op}}. 
\]
\end{lmm}
\begin{proof}
First, suppose that $e = \{\alpha, \beta\}$ is an edge such that $e\not\in A'$ and  $A = A'\cup \{e\}$. Since $A \cup B = A' \cup B'$ and $A\cap B = A'\cap B' = \emptyset$, this implies that $e\not \in B$ and $B'= B\cup \{e\}$. Thus,
\[
R(x,y,A,B)-R(x,y,A', B') = \sum_{q\in [N]^r}P(x,q,A')P(y,q,B)(x_{q_\alpha q_\beta} - y_{q_\alpha q_\beta}),
\]
and the proof is completed using Lemma \ref{comb1}. For the general case, simply `move' from the pair $(A, B)$ to the pair $(A',B')$ by `moving one edge at a time' and apply Lemma \ref{comb1} at each step. 
\end{proof}

\begin{lmm}\label{exupper22}
Let $g_{ij}$ denote the function $\partial T/\partial x_{ij}$, where $T$ is the function defined in equation \eqref{tdef}. Then for any $x,y\in [0,1]^n$,
\begin{align*}
\sum_{1\le i<j\le N} (g_{ij}(x)-g_{ij}(y))^2 &\le 8 m^2 k^2N\|x-y\|_{\textup{op}}\, . 
\end{align*}
\end{lmm}
\begin{proof}
Recall equation \eqref{tderiv}, that is, for any $1\le i<j\le N$, 
\begin{align*}
g_{ij}(x) = \fpar{T}{x_{ij}} &= \frac{1}{N^{k-2}}\sum_{\{a,b\}\in E}\sum_{\substack{q\in [N]^k\\ \{q_a,  q_b\} = \{i,j\}}} \prod_{\substack{\{l,l'\}\in E\\ \{l,l'\}\ne \{a, b\}}}x_{q_lq_{l'}}\, .  
\end{align*}
Although differentiating with respect to $x_{ii}$ does not make sense, let $g_{ii}$ be the function defined using the same formula as above. When $i>j$, let $g_{ij}=g_{ji}$. Fix $x,y\in [0,1]^n$. Define for any $q\in [N]^k$ and $\{a,b\}\in E$
\[
D(q,\{a,b\}) := \prod_{\substack{\{l,l'\}\in E\\ \{l,l'\}\ne \{a, b\}}}x_{q_lq_{l'}} - \prod_{\substack{\{l,l'\}\in E\\ \{l,l'\}\ne \{a, b\}}}y_{q_lq_{l'}}.
\]
Define
\begin{align*}
&\theta_{ij}  := 
\begin{cases}
2 & \text{ if } i=j,\\
1/2 & \text{ if } i\ne j,
\end{cases}
&\gamma_{ij} := 
\begin{cases}
2 & \text{ if } i=j,\\
1 & \text{ if } i\ne j.
\end{cases}
\end{align*}
Then note that
\begin{align*}
&\sum_{i,j=1}^N \theta_{ij}(g_{ij}(x)-g_{ij}(y))^2\\
&= \frac{1}{N^{2k-4}}\sum_{i,j=1}^N\theta_{ij}\biggl(\sum_{\{a,b\}\in E} \sum_{\substack{q\in [N]^k\\ \{q_a,q_b\}=\{i,j\}}} D(q, \{a,b\})\biggr)^2\\
&= \frac{1}{N^{2k-4}}\sum_{i,j=1}^N\sum_{\substack{\{a,b\}\in E\\ \{c,d\}\in E}} \sum_{\substack{q\in [N]^k\\ \{q_a, q_b\}=\{i,j\}}}\sum_{\substack{s\in [N]^k\\ \{s_c, s_d\}=\{i,j\}}}\theta_{ij}D(q, \{a,b\})D(s, \{c,d\})\\
&= \frac{1}{N^{2k-4}}\sum_{\substack{\{a,b\}\in E\\ \{c,d\}\in E}}\sum_{q\in [N]^k} \sum_{\substack{s\in [N]^k \\ \{s_c, s_d\}=\{q_a,q_b\}}}\gamma_{q_aq_b}D(q, \{a,b\})D(s, \{c,d\}). 
\end{align*}
Now fix two edges $\{a,b\}$ and $\{c,d\}$ in $E$. Relabeling vertices if necessary, assume that $c=k-1$ and $d=k$. Let $r = 2k-2$. For any $t\in [N]^r$, define two vectors $q(t)$ and $s(t)$ in $[N]^k$ as follows. For $i=1,\ldots,k$, let $q_i(t) = t_i$. For $i=1,\ldots, k-2$, let $s_i(t)=t_{i+k}$. Let $s_{k-1}(t) = t_a$ and $s_k(t)=t_b$. With this definition, it is clear that
\begin{align*}
&\sum_{q\in [N]^k} \sum_{\substack{s\in [N]^k\\ \{s_c, s_d\} = \{q_a, q_b\}}}\gamma_{q_aq_b}D(q, \{a,b\})D(s, \{c,d\})\\
&=\sum_{q\in [N]^k} \sum_{\substack{s\in [N]^k\\ s_c=q_a,\, s_d=q_b}}D(q, \{a,b\})D(s, \{c,d\}) \\
&\qquad + \sum_{q\in [N]^k} \sum_{\substack{s\in [N]^k\\ s_c=q_b,\, s_d=q_a}}D(q, \{a,b\})D(s, \{c,d\})\, .% \\
%&= 2\sum_{t\in [N]^r} D(q(t), \{a,b\})D(s(t), \{c,d\}).
\end{align*}
Note that the first term on the right-hand side is exactly equal to 
\[
\sum_{t\in [N]^r} D(q(t), \{a,b\})D(s(t), \{c,d\})\, .
\]
Below, we will get a bound on this term. The same upper bound will hold for the other term by symmetry.

Next, define two subsets of edges $A$ and $B$ of $K_r$ as follows. Let $A$ be the set of all edges $\{l,l'\}$ such that $\{l,l'\}\in E\backslash\{\{a,b\}\}$. Let $B$ be the set of all edges $\{\phi(l), \phi(l')\}$ such that $\{l,l'\}\in E\backslash\{\{k-1,k\}\}$, where $\phi:[k]\ra[r]$ is the map
\[
\phi(x) =
\begin{cases}
x+k &\text{ if $x\ne k-1$ and $x\ne k$,}\\
a &\text{ if $x=k-1$,}\\
b &\text{ if $x=k$.}
\end{cases}
\]
By the above construction, $q_l(t)=t_l$ and $s_l(t) = t_{\phi(l)}$. Therefore it is easy to see, for instance, that
\begin{align*}
&\sum_{t\in [N]^r}  \prod_{\substack{\{l,l'\}\in E\\ \{l,l'\}\ne \{a, b\}}}x_{q_l(t)q_{l'}(t)} \prod_{\substack{\{l,l'\}\in E\\ \{l,l'\}\ne \{k-1,k\}}}y_{s_l(t)s_{l'}(t)}\\
&= \sum_{t\in [N]^r}  \prod_{\substack{\{l,l'\}\in E\\ \{l,l'\}\ne \{a, b\}}}x_{t_lt_{l'}} \prod_{\substack{\{l,l'\}\in E\\ \{l,l'\}\ne \{k-1,k\}}}y_{t_{\phi(l)} t_{\phi(l')}}\\
&= R(x,y, A, B).
\end{align*}
Carrying out similar computations for the remaining terms in $D(q(t), \{a,b\})D(s(t), \{c,d\})$, we get
\begin{align*}
&\sum_{t\in [N]^r} D(q(t), \{a,b\})D(s(t), \{c,d\}) \\
&= R(x,y, A\cup B, \emptyset) - R(x,y,B, A) - R(x,y,A,B) + R(x,y,\emptyset, A\cup B). 
\end{align*}
Lastly, note that $A\cap B = \emptyset$ since for any $\{l,l'\}\in E\backslash \{\{k-1,k\}\}$, at least one among $\phi(l)$ and $\phi(l')$ must be strictly bigger than $k$ and therefore $\{\phi(l),\phi(l')\}$ cannot be an element of $A$. The proof is now easily completed by applying Lemma \ref{comb2}.
\end{proof}
With the help of Lemma \ref{exupper21} and Lemma \ref{exupper22}, we are now ready to prove Lemma \ref{test2}.
\begin{proof}[Proof of Lemma \ref{test2}]
Take any $\ep>0$ and let 
\[
\tau = \frac{\ep^2}{64m^2k^2}\, .
\]
Let $\mw(\tau)$ be as in Lemma \ref{exupper21}. For each $W\in \mw(\tau)$, let $y(W)\in [0,1]^n$ be a vector such that $\|M(y)-W\|_{\textup{op}}\le N\tau$. If for some $W$ there does not exist any such $y$, leave $y(W)$ undefined. Let $g_{ij} = \partial T/\partial x_{ij}$, as in Lemma~\ref{exupper22}. Let $g:[0,1]^n \ra\rr^n$ be the function whose $(i,j)$th coordinate is $g_{ij}$. Define
\begin{align*}
\dd(\ep) &:= \big\{g(y) : \text{$y= y(W)$ for some $W\in \mw(\tau)$}\big\}.  
\end{align*}
Then by Lemma \ref{exupper21}
\begin{align*}%\label{desize}
|\dd(\ep)|&\le |\mw(\tau)| \le e^{34(N/\tau^2)\log(51/\tau^2)}. 
\end{align*}
We claim that the set $\dd(\ep)$ satisfies the requirements of Theorem \ref{uppertailthm}. To see this, take any $x\in [0,1]^n$. By Lemma \ref{exupper21}, there exists $W\in \mw(\tau)$ such that $\|M(x) - W\|_{\textup{op}} \le N\tau$. In particular, this means that $y := y(W)$ is defined, and so
\begin{align*}
\|x-y\|_{\textup{op}} &= \|M(x)-M(y)\|_{\textup{op}} \\
&\le  \|M(x)-W\|_{\textup{op}}+ \|W-M(y)\|_{\textup{op}}\\
&\le 2N\tau\, .
\end{align*}
Therefore by Lemma \ref{exupper22}, 
\begin{align*}%\label{expart1}
\sum_{1\le i<j\le N} (g_{ij}(x)-g_{ij}(y))^2 &\le 16m^2 k^2N^2\tau\, . 
\end{align*}
Let $z= g(x)$ and $v = g(y)$. Then $v\in \dd(\ep)$, and by the above inequality,
\begin{align*}
\sum_{1\le i<j\le N} (z_{ij} - v_{ij})^2  &\le  16 m^2k^2N^2\tau = \frac{N^2\ep^2}{4}\le {N\choose 2} \ep^2.
\end{align*}
This proves the claim that $\dd(\ep)$ satisfies the requirements of Theorem \ref{uppertailthm}. This completes the proof of Lemma \ref{test2}.
\end{proof}
The next step is to understand the properties of the rate function $\phi_p(t)$ corresponding to $T$. First, we need a simple lemma.
\begin{lmm}\label{abineq}
For any $r$ and any $a_1,\ldots, a_r, b\in [0,1]$, 
\begin{equation*}
\prod_{i=1}^r (a_i + b(1-a_i)) \ge (1-b^r) \prod_{i=1}^r a_i + b^r\,.
\end{equation*}
\end{lmm}
\begin{proof}
The proof is by induction on $r$. The inequality is an equality for $r=1$. Suppose that it holds for $r-1$. Then %by the nonnegativity of all terms,
\begin{align*}
\prod_{i=1}^r (a_i + b(1-a_i)) &\ge \biggl((1-b^{r-1}) \prod_{i=1}^{r-1} a_i + b^{r-1}\biggr)((1-b)a_r + b)\\
&= (1-b^{r-1})(1-b) \prod_{i=1}^{r} a_i + b^{r-1}(1-b)a_r + (1-b^{r-1})b \prod_{i=1}^{r-1} a_i+b^r\\
&\ge ((1-b^{r-1})(1-b) + b^{r-1} (1-b) + (1-b^{r-1})b)\prod_{i=1}^r a_i + b^r\\
&= (1-b^r)\prod_{i=1}^r a_i + b^r\,.
\end{align*}
This completes the induction.
\end{proof}
\begin{lmm}\label{lipschitz}
Let $\phi_p(t)$ be defined as in \eqref{phidef}, with $f=T$ and $n= N(N-1)/2$. Let $l$ be the element of $[0,1]^n$ whose coordinates are all equal to $1$, and let $t_0:= T(l)/n$. Then for any $0<\delta< t< t_0$, 
\begin{align*}
\phi_p(t-\delta) &\ge \phi_p(t)-\biggl(\frac{\delta}{t_0-t}\biggr)^{1/m}n \log (1/p)\, .
\end{align*}
\end{lmm}
\begin{proof}
%If $tn > \|T\|$, then $\phi_p(t)=-\infty$ since there is no $x$ that satisfies $T(x)> \|T\|$. Therefore assume that $tn \le \|T\|$. Note that $T(l)=\|T\|$. 
Take any $x\in [0,1]^n$ such that $T(x) \ge (t-\delta)n$ and $x$ minimizes $I_p(x)$ among all $x$ satisfying this inequality. If $T(x)\ge tn$, then we immediately have $\phi_p(t)\le I_p(x)=\phi_p(t-\delta)$, and there is nothing more to prove. So let us assume that $T(x)<tn$. Let 
 \[
 \ep := \biggl(\frac{tn - T(x)}{T(l)-T(x)}\biggr)^{1/m}\, .
 \]
For each $1\le i<j\le N$, let  
\[
y_{ij} := x_{ij} + \ep(1-x_{ij})\, .  
\] 
Let $y_{ji}=y_{ij}$ and $y_{ii}=0$. Then $y\in [0,1]^n$, and by Lemma \ref{abineq},
\begin{align*}
T(y) &\ge (1-\ep^{m}) T(x) + \ep^{m} T(l) = tn\, .
\end{align*}
Thus, by the convexity of $I_p$,
\begin{align*}
\phi_p(t) &\le I_p(y) = I_p((1-\ep)x + \ep l)\\
&\le (1-\ep)I_p(x) + \ep I_p(l)\\
&\le I_p(x) + \ep\, n \log (1/p) = \phi_p(t-\delta) + \ep\, n \log (1/p)\, .
\end{align*}
Since $T(x)\ge (t-\delta)n$, %$\|T\|\ge N(N-1)\cdots (N-k+1)\ge N^k/2$, and 
\begin{align*}
\ep^{m} &\le  \frac{tn - (t-\delta)n}{T(l)-(t-\delta)n} \le \frac{\delta }{t_0-t}\,  .
\end{align*}
This completes the proof of the lemma. 
\end{proof}
\begin{lmm}\label{phiupper}
For any $p$ and $t$,
\[
\phi_p(t)\le \frac{1}{2}(\lceil t^{1/k}N\rceil+k)^2\log(1/p)\, .
\]
\end{lmm}
\begin{proof}
Let $r := \lceil t^{1/k}N\rceil+k$. Define $x\in [0,1]^n$ as
\[
x_{ij} := 
\begin{cases}
1 & \text{ if } 1\le i<j\le r\, ,\\
p & \text{ otherwise.}
\end{cases}
\]
Then 
\begin{align*}
T(x) &\ge \frac{1}{N^{k-2}} \sum_{q\in [r]^k} \prod_{\{l,l'\}\in E} x_{q_lq_{l'}} \\
&\ge \frac{r(r-1)\cdots(r-k+1)}{N^{k-2}} \ge tN^2\ge tn\, , 
\end{align*}
and since $I_p(p)=0$, 
\begin{align*}
I_p(x) &= \sum_{i<j} I_p(x_{ij}) \le \frac{1}{2} r^2 \log(1/p)\, .
\end{align*}
This proves the claim.
\end{proof}
\begin{proof}[Proof of the upper bound in Theorem \ref{graphthm}]
The task now is to pull together all the information obtained above, for use in Theorem \ref{uppertailthm}. As intended, we work with $f=T$. Take $t = \kappa p^m$ for some fixed $\kappa>0$. Let $\delta$ and $\ep$ be two positive real numbers, both less than $t$, to be chosen later.  Note that $\delta < t<  \kappa p^{2m/k}$ since $t = \kappa p^m$ and $k>2$. Assume that $\delta$ and $\ep$ are bigger than $N^{-1/2}$. Note that $p$ is already assumed to be bigger than $N^{-1/2}$ in the statement of the theorem.

%Throughout this proof, $C$ will denote any constant that may depend only on the graph $H$. 

 Recall that the indexing set for quantities like $b_i$ and $c_{ij}$, instead of being $\{1,\ldots,n\}$, is now $\{(i,j): 1\le i<j\le N\}$. For simplicity, we will write $(ij)$ instead of $(i,j)$. Throughout, $C$ will denote any constant that depends only on the graph $H$, the constant $\kappa$, and nothing else. From Lemma~\ref{test1}, we have the estimates
\[
a\le N^2\, , \ \ b_{(ij)} \le C\, , 
\] 
and
\[
c_{(ij)(i'j')} \le 
\begin{cases}
C N^{-1} \ &\text{ if $|\{i,j,i',j'\}| = 2$ or $3$\, ,} \\
C N^{-2} & \text{ if $|\{i,j,i',j'\}| = 4$\, .}
\end{cases} 
\]
%Let 
%\[
%r_1 := \frac{2m}{k}-r\, , \ \  r_2 := \frac{2m}{k}-2r\, .
%\]
Let $\theta := \delta^{-1} p^{2m/k}$. By Lemma \ref{phiupper}, 
\[
K \le Cp^{2m/k} \log N\, . 
\]
Using the above bounds, we get
\[
\alpha \le CN^2 \log N\, , \ \ \beta_{(ij)} \le C \theta \log N\, ,
\]
and 
\begin{align*}
\gamma_{(ij)(i'j')} %&\le CN^{-2} \delta^{-1}\theta \log N  \\
%&\qquad + 
&\le 
\begin{cases}
C N^{-1} \theta\log N  &\text{ if $|\{i,j,i',j'\}| = 2$ or $3$,} \\
C N^{-2} \delta^{-1}\theta\log N & \text{ if $|\{i,j,i',j'\}| = 4$.}
\end{cases} 
\end{align*}
Therefore, we have the estimates
\begin{align*}
\sum_{(ij)} \beta_{(ij)}^2 &\le CN^2\theta^2 (\log N)^2\, , \ \ \sum_{(ij)} b_{(ij)}^2 \le CN^2\, ,
\end{align*}
and by Lemma \ref{test2}, 
\begin{align*}
\log|\dd((\delta\ep)/(4K))| &\le \frac{CN\theta^4}{ \ep^4}\log \frac{CK}{\delta \ep}\\
&\le \frac{CN\theta^4(\log N)^5}{\ep^4} \, .
\end{align*}
Combining the last three estimates, we see that the complexity term in Theorem \ref{uppertailthm} is bounded above by
\begin{align*}
&CN^2 \ep\theta\log N +  \frac{CN\theta^4(\log N)^5}{\ep^4}\, .
\end{align*}
Taking $\ep =  N^{-1/5}\theta^{3/5}(\log N)^{4/5}$, the above bound simplifies to
\begin{align*}
CN^{9/5} \theta^{8/5} (\log N)^{9/5}\, .
\end{align*}
Next, note that by the bounds obtained above and the inequality $\delta > N^{-1/2}$, 
\begin{align*}
&\sum_{(ij)} \alpha \gamma_{(ij)(ij)} \le CN^3\theta (\log N)^2\, ,\\
%\end{align*}
%\begin{align*}
&\sum_{(ij), (i'j')} \alpha \gamma_{(ij)(i'j')}^2 \le %CN^2(\log N) (N^3 N^{-2}\theta^2 \log^2 N + N^4 N^{-4} \delta^{-2}\theta^2 \log^2 N)\\
%&\le CN^2(N + \delta^{-2}) \theta^2 \log^3 N\le 
CN^3 \theta^2 (\log N)^3\, ,\\
%\end{align*}
%\begin{align*}
&\sum_{(ij), (i'j')} \beta_{(ij)} (\beta_{(i'j')} +4)\gamma_{(ij)(i'j')} %\\
%&\le C \theta^2 (\log N)^2(N^3 N^{-1}\theta \log N + N^4 N^{-2}\delta^{-1} \theta \log N)\\
%&
\le CN^2 \delta^{-1}\theta^3 (\log N)^3\, ,\\
%\end{align*}
%\begin{align*}
&\biggl(\sum_{(ij)} \beta_{(ij)}^2\biggr)^{1/2}\biggl(\sum_{(ij)} \gamma_{(ij)(ij)}^2\biggr)^{1/2} \le CN\theta^2 (\log N)^2\, ,\\
%\end{align*}
%and 
%\begin{align*}
&\sum_{(ij)} \gamma_{(ij)(ij)} \le CN\theta \log N\, . 
\end{align*}
The above estimates show that the smoothness term in Theorem \ref{uppertailthm} is bounded above by a constant times
\begin{align*}
N^{3/2}\theta (\log N)^{3/2} + N\delta^{-1/2}\theta^{3/2}(\log N)^{3/2} + N\theta^2 (\log N)^2\, .
\end{align*}
Putting $\eta := p^{2m/k}$,  and recalling that $N^{-1/(m+3)}\le p\le 1-N^{-1}$, we see that this is bounded by a constant times 
\[
N^{3/2} \delta^{-1}\eta (\log N)^{3/2} + N\delta^{-2}\eta^{3/2}(\log N)^2 \, .
\]
Since $\delta > N^{-1/2}$, we can further simplify this upper bound to
\[
N^{3/2}\delta^{-1}\eta(\log N)^2. 
\]
Combining the bounds on the complexity term and the smoothness term, we get that 
\begin{align*}
\log\pp(T(Y) \ge tn) &\le -\phi_p(t-\delta) + CN^{9/5} \delta^{-8/5} \eta^{8/5} (\log N)^{9/5} \\
&\qquad + C N^{3/2} \delta^{-1}\eta (\log N)^2\, . 
\end{align*}
By Lemma \ref{lipschitz}, 
\[
-\phi_p(t-\delta) \le -\phi_p(t) + C\delta^{1/m} N^2\log N. 
\]
Taking
\[
\delta = N^{-m/(5+8m)} \eta^{8m/(5+8m)} (\log N)^{4m/(5+8m)}
\]
gives
\begin{align}
&\log\pp(T(Y)\ge tn) \label{compl}\\
&\le -\phi_p(t) + CN^{(9+16m)/(5+8m)} \eta^{8/(5+8m)}(\log N)^{(9+8m)/(5+8m)}\nonumber\\
&\qquad + CN^{(15+26m)/(10+16m)} \eta^{5/(5+8m)} (\log N)^{(10+12m)/(5+8m)}\, .\nonumber
\end{align}
Now note that since $p > N^{-1/2}$, therefore
\begin{align*}
\frac{N^{(9+16m)/(5+8m)} \eta^{8/(5+8m)}}{N^{(15+26m)/(10+16m)} \eta^{5/(5+8m)}} &= N^{(3+6m)/(10+16m)} p^{6m/k(5+8m)}\\
&\ge N^{(3+6m)/(10+16m)} N^{-3m/(5+8m)}\\
&= N^{3/(10+16m)}\, .
\end{align*}
This shows that the first term on the right-hand side in \eqref{compl} dominates the second when $N$ is sufficiently large. Therefore, when $N$ is large enough,
\begin{align*}
&\log \pp(T(Y)\ge tn)\\
 &\le -\phi_p(t) + CN^{(9+16m)/(5+8m)} p^{16m/k(5+8m)}(\log N)^{(9+8m)/(5+8m)}\, .
\end{align*}
Written differently, this is
\begin{align*}
&\frac{\phi_p(t)}{-\log\pp(T(Y)\ge tn)}\\
&\le 1 + \frac{CN^{(9+16m)/(5+8m)} p^{16m/k(5+8m)}(\log N)^{(9+8m)/(5+8m)}}{-\log\pp(T(Y)\ge tn)}\, .
\end{align*}
By  \cite[Theorem 1.2 and Theorem 1.5]{JOR04}, 
\begin{equation}\label{janson}
-\log \pp(T(Y)\ge tn) \ge C N^2 p^\Delta\, , 
\end{equation}
where $\Delta$ is the maximum degree of $H$, provided that $p\ge N^{-1/\Delta}$ and $N$ is sufficiently large. The lower bound on $p$ is already assumed in the statement of the theorem. Therefore, 
\begin{align*}
&\frac{\phi_p(t)}{-\log\pp(T(Y)\ge tn)}\\
&\le 1 + CN^{-1/(5+8m)} p^{-\Delta + 16m/k(5+8m)}(\log N)^{(9+8m)/(5+8m)}\, .
\end{align*}
A minor verification using the assumption $p\ge N^{-1/4}$ shows that the $\ep$ and $\delta$ chosen above are both bigger than $N^{-1/2}$, as required. To complete the proof of the upper bound, notice that $\ee(X)$ is asymptotic to $p^m$ since $p\ge N^{-1/(m+3)}$. 
\end{proof}
\begin{proof}[Proof of the lower bound in Theorem \ref{graphthm}]
By Lemma \ref{lipschitz}, Lemma \ref{test1}, and the lower bound in Theorem~\ref{uppertailthm},  
\begin{align*}
\log \pp(T(Y)\ge tn) &\ge -\phi_p(t) - C N^{-1/2m}N^2\log N\, .
\end{align*}
Therefore, again applying \eqref{janson}, we get
\begin{align*}
\frac{\phi_p(t)}{-\log\pp(T(Y)\ge tn)} \ge 1 - C N^{-1/2m} p^{-\Delta} \log N\, .
\end{align*}
This completes the proof of the lower bound.
\end{proof}

\section{Proof of Theorem \ref{arith}}
In this section, all indices range over $\zz/n\zz$, and all additions and subtractions of indices are modulo $n$. As usual, $C$ will denote any universal constant.

Let $Y = (Y_0,\ldots, Y_{n-1})$ be a vector of i.i.d.\ $Bernoulli(p)$ random variables. Define $f:[0,1]^{\zz/n\zz} \ra\rr$ as 
\[
f(x) := \frac{1}{n}\sum_{i,j} x_i x_{i+j} x_{i+2j}\, .
\]
Then 
\begin{equation}\label{afbd}
a:= \|f\|\le n\, .
\end{equation}
Let $f_i := \partial f/\partial x_i$ and $f_{ij} := \partial^2 f /\partial x_i \partial x_j$. Then
\begin{align*}
f_i(x) &= \frac{1}{n} \sum_{j} (x_{i+j} x_{i+2j} + x_{i-j} x_{i+j} + x_{i-2j} x_{i-j})\, .
\end{align*} 
From this expression, it is clear that
\begin{equation}\label{afbd2}
b_i := \|f_i\|\le C\, , \ \ c_{ij}:= \|f_{ij}\|\le \frac{C}{n}\, .
\end{equation}
For each $j$, define the function $e_j:\zz/n\zz\ra\cc$ as
\[
e_j(k) := \frac{1}{\sqrt{n}}e^{2\pi\I jk/n}\, ,
\]
where $\I = \sqrt{-1}$. These functions form an orthonormal system, in the sense that
\begin{align*}
\sum_k e_j(k)\overline{e_{j'}(k)} = \delta_{j-j'}\, ,
\end{align*}
where $\delta$ is the Kronecker delta function, that is,
\begin{align*}
\delta_j := 
\begin{cases}
1 & \text{ if } j=0\, ,\\
0 & \text{ otherwise.}
\end{cases}
\end{align*}
For any $x\in\rr^{\zz/n\zz}$, define its discrete Fourier transform $\hx\in \cc^n$ as 
\[
\hx_j := \sum_k x_k e_j(k)\, .
\]
The orthonormality of the $e_j$'s implies the inversion formula
\begin{align*}
\sum_j \hx_j \overline{e_k(j) } &= \sum_{j,l}x_le_j(l)  \overline{e_k(j) }\\
&= \sum_{j,l}x_l e_l(j)  \overline{e_k(j) }=x_k\, .
\end{align*}
Moreover, it also implies the Plancherel identity
\begin{align*}
\sum_j |\hx_j|^2 &= \sum_{j,k,l} x_kx_l e_j(k)\overline{e_j(l)} \\
&= \sum_{j,k,l} x_kx_l e_k(j)\overline{e_l(j)} = \sum_k x_k^2\, .
\end{align*}
\begin{lmm}\label{fourier}
For any $x, y\in [0,1]^{\zz/n\zz}$, 
\[
\sum_i (f_i(x)-f_i(y))^2 \le Cn^{1/2}\max_i |\hx_i - \hat{y}_i|\, .
\]
\end{lmm}
\begin{proof}
Note that for any $x$ and $y$,
\begin{align}
&\sum_{i} (f_i(x)-f_i(y))^2 \label{fixfiy}\\
&= \frac{1}{n^2}\sum_{i} \biggl(\sum_{j} (x_{i+j}x_{i+2j}+ x_{i-j} x_{i+j} + x_{i-2j} x_{i-j} \nonumber\\
&\qquad \qquad  - y_{i+j}y_{i+2j}- y_{i-j} y_{i+j} - y_{i-2j} y_{i-j})\biggr)^2 \nonumber\\
&= \frac{1}{n^2}\sum_{i,j,k} (x_{i+j}x_{i+2j}+ x_{i-j} x_{i+j} + x_{i-2j} x_{i-j} \nonumber\\
&\qquad \qquad - y_{i+j}y_{i+2j}- y_{i-j} y_{i+j} - y_{i-2j} y_{i-j})\nonumber\\
&\qquad \qquad \times (x_{i+k}x_{i+2k}+ x_{i-k} x_{i+k} + x_{i-2k} x_{i-k}\nonumber\\
&\qquad \qquad \qquad - y_{i+k}y_{i+2k}- y_{i-k} y_{i+k} - y_{i-2k} y_{i-k})\, .\nonumber
\end{align}
Let us now expand out the product in the above expression. There will be 36 terms, 18 of which are positive and 18 are negative. The positive terms will be products of fours $x$'s or four $y$'s, and the negative terms will be products of two $x$'s and two $y$'s. Match each positive term with a matching negative term. For example, match $x_{i+j}x_{i+2j} x_{i-k} x_{i+k}$ with $-x_{i+j}x_{i+2j} y_{i-k}y_{i+k}$. Summing over $i$, $j$ and $k$ for this particular pair, we get the expression 
\begin{align}
&\frac{1}{n^2}\sum_{i,j,k} (x_{i+j}x_{i+2j} x_{i-k} x_{i+k}-x_{i+j}x_{i+2j} y_{i-k}y_{i+k})\label{fixfiy1}\\
&= \frac{1}{n^2}\sum_{i,j,k} (x_{i+j}x_{i+2j} (x_{i-k}-y_{i-k}) x_{i+k}-x_{i+j}x_{i+2j} y_{i-k}(x_{i+k}-y_{i+k}))\, .\nonumber
\end{align}
Now consider the first term in the above expression. Let $z=x-y$. Then by the inversion formula,
\begin{align*}
& \frac{1}{n^2}\sum_{i,j,k} x_{i+j}x_{i+2j} x_{i+k}z_{i-k}\\
 &= \frac{1}{n^2}\sum_{i,j,k}\sum_{a,b,c,d} \hx_a \hx_b \hx_c \hat{z}_d\overline{e_{i+j}(a) e_{i+2j}(b)e_{i+k}(c)e_{i-k}(d)}\\
 &= \frac{1}{n^{5/2}}\sum_{a,b,c,d}\sum_{i,j,k} \hx_a \hx_b \hx_c \hat{z}_d\overline{e_{a+b+c+d}(i) e_{a+2b}(j)e_{c-d}(k)}\\
 &= \frac{1}{n}\sum_{a,b,c,d} \hx_a \hx_b \hx_c \hat{z}_d\delta_{a+b+c+d} \delta_{a+2b} \delta_{c-d}= \frac{1}{n}\sum_{d} \hx_{-4d} \hx_{2d} \hx_d \hat{z}_d\, .
\end{align*}
%In the above sum, the four indices $a$, $b$, $c$ and $d$ satisfy three linear equations between themselves, due to the presence of the three delta functions. Therefore there is actually only one degree of freedom. In particular, there are linear maps $l_1$, $l_2$, $l_3$ and $l_4$ from $\zz/n\zz$ into itself,  such that
%\[
%\sum_{a,b,c,d} \hx_a \hx_b \hx_c \hat{z}_d\delta_{a+b+c+d} \delta_{a+2b} \delta_{c-d} = \sum_{i} \hx_{l_1(i)} \hx_{l_2(i)} \hx_{l_3(i)} \hat{z}_{l_4(i)}\, .
%\]
%For example, taking $l_1(i) = -4i$, $l_2 (i) = 2i$ and $l_3(i) = l_4(i)=i$ does the job. 
By the Plancherel identity and the fact that $x\in [0,1]^{\zz/n\zz}$, $\sum_j |\hx_j|^2 \le n$. In particular, $|\hx_j|\le \sqrt{n}$ for all $j$. Let $M:=\max_i |\hat{z}_i|$. Using these observations and H\"older's inequality, we see that the above sum is bounded above by 
\begin{align*}
&\frac{M}{n} \biggl(\sum_d|\hx_{-4d}|^3\sum_d|\hx_{2d}|^3\sum_d|\hx_{d}|^3\biggr)^{1/3}\\
&\le \frac{CM}{n}\sum_d|\hx_{d}|^3\le CMn^{1/2}\, . 
\end{align*}
This is a bound on the first term in the right-hand side of \eqref{fixfiy1}. Similarly, it may be verified that the same bound holds for the second term in the right-hand side of \eqref{fixfiy1}, and also for all terms in the expansion of \eqref{fixfiy}. This completes the proof of the lemma.
\end{proof}
\begin{lmm}\label{ade}
For the function $f$ considered in this section, one can find sets $\dd(\ep)$ satisfying \eqref{entropy} such that $|\dd(\ep)|\le C_1(n/\ep^2)^{C_2/\ep^4}$ where $C_1$ and $C_2$ are universal constants. 
\end{lmm}
\begin{proof}
Take any $\ep>0$. Let $\gamma := c \ep^2 \sqrt{n}$, where $c$ is a universal constant that will be chosen later. Define a map $R: \cc^n \ra \cc^n$ as follows: For each $i$, let the $i$th coordinate of $y = R(x)$ be the complex number closest to $x_i$ whose real and imaginary parts are both integer multiples of $\gamma$. Clearly, $|x_i-y_i|\le \gamma$. Moreover, if $|x_i|< \gamma/2$ then $y_i = 0$. 

Let $\mm$ be the set of all $\hx$ as $x$ ranges over $[0,1]^n$. Take any $x\in [0,1]^n$ and let $y := R(\hx)$.  Let $A$ be the set of all $i$ such that $|\hx_i|\ge \gamma/2$. Then $y_i = 0$ for each $i\not\in A$. Given $A$, there are at most $Cn/\gamma^2$ possible values of each $y_i$, $i\in A$ since $|\hx_i|\le \sqrt{n}$. On the other hand by the Plancherel identity, 
\begin{align*}
|A|&\le \frac{4}{\gamma^2}\sum_{i=0}^{n-1} |\hx_i|^2\le \frac{4n}{\gamma^2}\, ,
\end{align*}
implying that there are at most $n^{4n/\gamma^2}$ possible candidates  for the set $A$. Combining these observations, we realize that the number of possible values of $y$ is at most 
\[
n^{4n/\gamma^2} (Cn/\gamma^2)^{4n/\gamma^2}\, .% = 4 n^{8n/\gamma^2}\, .
\]
This, therefore, is a bound on the size of $R(\mm)$. 

Say that two points $x$ and $y$ in $[0,1]^n$ are equivalent if $R(\hx)=R(\hat{y})$. Clearly, this is an equivalence relation. Suppose that $x$ and $y$ are equivalent. Let $z=R(\hx)=R(\hat{y})$. Then for each $i$,
\begin{align*}
|\hx_i-\hat{y}_i|&\le |\hx_i-z_i|+|z_i-\hat{y}_i|\le 2\gamma\, .
\end{align*}
Construct the set $B$ by choosing one $x$ from each equivalence class. Then clearly
\[
|B|\le |R(\mm)|\le n^{4n/\gamma^2} (Cn/\gamma^2)^{4n/\gamma^2}\, .
\]
By the bounds obtained above and Lemma \ref{fourier}, for any $x\in [0,1]^n$, there exists $y\in B$ such that 
\begin{align*}
\sum_{i} (f_i(x)-f_i(y))^2 &\le Cn^{1/2}\max_i|\hx_i - \hat{y}_i| \le Cn^{1/2} \gamma\, .% \le \ep^2 n\, .
\end{align*}
The right-hand side is less than $\ep^2 n$ if the constant $c$ in the definition of $\gamma$ is chosen sufficiently small. Defining $\dd(\ep)$ to be the set $\nabla f (B)$ completes the proof. 
\end{proof}
\begin{lmm}\label{alip}
Let $\phi_p(t)$ be defined as in \eqref{phidef}. Then for any $0<\delta< t< 1$, 
\begin{align*}
\phi_p(t-\delta) &\ge \phi_p(t)-\biggl(\frac{\delta}{1-t}\biggr)^{1/3}n \log (1/p)\, .
\end{align*}
\end{lmm}
\begin{proof}
%If $tn > \|T\|$, then $\phi_p(t)=-\infty$ since there is no $x$ that satisfies $T(x)> \|T\|$. Therefore assume that $tn \le \|T\|$. Note that $T(l)=\|T\|$. 
Take any $x\in [0,1]^n$ such that $f(x) \ge (t-\delta)n$ and $x$ minimizes $I_p(x)$ among all $x$ satisfying this inequality. If $f(x)\ge tn$, then we immediately have $\phi_p(t)\le I_p(x)=\phi_p(t-\delta)$, and there is nothing more to prove. So let us assume that $f(x)<tn$. Let 
 \[
 \ep := \biggl(\frac{tn - f(x)}{n-f(x)}\biggr)^{1/3}\, .
 \]
For each $i$, let  
\[
y_{i} := x_{i} + \ep(1-x_{i})\, .  
\]  
Then $y\in [0,1]^n$, and by Lemma \ref{abineq},  we get
\begin{align*}
f(y) &\ge (1-\ep^{3}) f(x) + \ep^{3} = tn\, .
\end{align*}
Thus, by the convexity of $I_p$,
\begin{align*}
\phi_p(t) &\le I_p(y) \le (1-\ep)I_p(x) + \ep\, n \log (1/p)\\
&\le I_p(x) + \ep\, n \log (1/p) = \phi_p(t-\delta) + \ep\, n \log (1/p)\, .
\end{align*}
Since $f(x)\ge (t-\delta)n$, %$\|T\|\ge N(N-1)\cdots (N-k+1)\ge N^k/2$, and 
\begin{align*}
\ep^{3} &\le  \frac{tn - (t-\delta)n}{n-(t-\delta)n} \le \frac{\delta }{1-t}\,  .
\end{align*}
This completes the proof of the lemma. 
\end{proof}
\begin{lmm}\label{aklmm}
For any $p\ge n^{-1}$ and $t> 0$,
\[
\phi_p(t)\le C t^{1/2}n \log n\, .
\]
\end{lmm}
\begin{proof}
Define $x\in [0,1]^n$ as the vector whose first $3t^{1/2}n$ coordinates are equal to $1$ and the rest are equal to $p$. 
Then 
\begin{align*}
f(x) &= \frac{1}{n} \sum_{i,j} x_ix_{i+j}x_{i+2j} \ge tn\, ,
\end{align*}
and $I_p(x) \le Ct^{1/2} n \log n$. This proves the claim.
\end{proof}

\begin{lmm}\label{aconc}
Suppose that $p\ge n^{-1/6}$. Then for any $\kappa > 1$, 
\[
\pp(f(Y)\ge \kappa p^3 n) \le Ce^{-cn p^6}
\]
where $C$ and $c$ depend only on $\kappa$. 
\end{lmm}
\begin{proof}
Let $\kappa' := (1+\kappa)/2$, so that $1< \kappa'< \kappa$. It is easy to see that if $n$ is sufficiently large (depending on $\kappa$), then $\ee(f(Y))\le \kappa' p^3 n$. Again, \eqref{afbd2} shows that $f(Y)$ changes at most by a bounded amount if one $Y_i$ changes value. Therefore a straightforward application of Hoeffding's inequality \cite{hoeffding63} completes the proof of the lemma. 
\end{proof}
\begin{proof}[Proof of Theorem \ref{arith}]
Let $0<\delta < t<1$, where $t = \kappa p^3$ for some $\kappa > 1$ and $\delta$ is to be chosen later. Fix another small quantity $\ep$, also to be chosen later. Assume that $\ep$ and $\delta$ are both bigger than $n^{-1/3}$. Already from the statement of the theorem, recall that $p\ge n^{-1/162}$. 

Throughout this proof $C$ will denote any constant that may depend only on~$\kappa$. Let $K$, $\alpha$, $\beta_i$ and $\gamma_{ij}$ be defined as in Theorem \ref{uppertailthm}. By \eqref{afbd}, \eqref{afbd2}, Lemma \ref{aklmm} and the assumption that $n^{-1/162}\le p\le 1-n^{-1}$, 
\[
K\le Ct^{1/2}\log n\, , \ \ \alpha\le Cn \log n\, , \ \ \beta_i \le \frac{Ct^{1/2}\log n}{\delta}\, , \ \ \gamma_{ij} \le  \frac{Ct^{1/2}\log n}{\delta^2 n}\, .
\]
These imply the bounds
\begin{align*}
&\sum_i \alpha \gamma_{ii} \le \frac{Ct^{1/2}n(\log n)^2}{\delta^2}\le \frac{Ct n (\log n)^2}{\delta^{5/2}}\, ,\\
&\sum_i \beta_i^2 \le \frac{Ctn(\log n)^2}{\delta^2} \, , \ \ \sum_{i,j} \alpha \gamma_{ij}^2 \le \frac{C tn (\log n)^3}{ \delta^{4}}\, ,\\
&\sum_{i,j}\beta_i(\beta_j + 4) \gamma_{ij} \le\frac{Ct^{3/2} n (\log n)^3}{ \delta^{4}} \, , \\
&\sum_i \gamma_{ii}^2  \le \frac{Ct(\log n)^2 }{\delta^4 n}\, , \ \ \sum_i \gamma_{ii}\le \frac{Ct^{1/2}\log n}{\delta^2}\, . 
\end{align*}
Combining these, we see that the smoothness term is bounded by
\begin{align*}
&C t^{1/2} \delta^{-2}n^{1/2}(\log n)^{3/2}+ C t \delta^{-3}(\log n)^2 + C t^{1/2}\delta^{-2}\log n\, .
\end{align*}
Since $\delta > n^{-1/3}$, the above expression is bounded by 
\begin{align}\label{asmooth}
C t^{1/2} \delta^{-2}n^{1/2}(\log n)^{3/2}\, .
\end{align}
On the other hand by Lemma \ref{ade} and the assumption that $\ep > n^{-1/3}$, 
\begin{align*}
\log |\dd(\ep)|\le \frac{C\log (n/\ep^2)}{\ep^4}\le \frac{C\log n}{\ep^4}\, .
\end{align*}
Therefore the complexity term is bounded above by 
\begin{align*}
C\ep t^{1/2}\delta^{-1}n \log n + \log\biggl(\frac{Ct^{1/2}\log n}{\delta \ep}\biggr) + \frac{Ct^2(\log n)^5}{\delta^4\ep^4}\, . 
\end{align*}
Choosing 
\[
\ep = t^{3/10}\delta^{-3/5}n^{-1/5}(\log n)^{4/5}
\]
and recalling the assumed lower bounds on $\ep$, $\delta$ and $p$, we see that the complexity term is bounded by
\begin{align}\label{acomplex}
Ct^{4/5}\delta^{-8/5}n^{4/5}(\log n)^{9/5}\, .
\end{align}
By Theorem \ref{uppertailthm}, Lemma \ref{alip}, the bound \eqref{asmooth} for the smoothness term, and the bound \eqref{acomplex} for the complexity term, we get that
\begin{align*}
\log \pp(f(Y)\ge tn) &\le -\phi_p(t) + C\delta^{1/3} n\log n + Ct^{4/5}\delta^{-8/5}n^{4/5}(\log n)^{9/5}\\
&\qquad  + C t^{1/2} \delta^{-2}n^{1/2}(\log n)^{3/2}\, .
\end{align*}
Now choose
\[
\delta = t^{12/29} n^{-3/29} (\log n)^{12/29}\, .
\]
Recalling that $t=\kappa p^3$, this gives
\begin{align*}
\log \pp(f(Y)\ge tn) &\le -\phi_p(t) + C p^{12/29} n^{28/29} (\log n)^{33/29} \\
&\qquad + Cp^{-57/58}n^{41/58}(\log n)^{39/58}\, .
\end{align*}
By the assumed lower bound on $p$, it is easy to see that the second term on the right dominates the third if $n$ is large enough. Together with an application of Lemma \ref{aconc}, this completes the proof of the upper bound in Theorem \ref{arith}. For the lower bound, recall that 
\begin{align*}
\log \pp(f(Y)\ge tn) &\ge -\phi_p(t+\delta_0) - \ep_0 n - \log 2\\
&\ge -\phi_p(t) - C\delta_0^{1/3}n \log n - \ep_0 n\, ,
\end{align*}
where
\[
\ep_0  = \frac{1}{\sqrt{n}}\biggl(4+\biggl|\log\frac{p}{1-p}\biggr|\biggr)\le Cn^{-1/2}\log n
\]
and
\[
\delta_0 = \frac{2}{n}\biggl(\sum_{i=1}^n (ac_{ii} + b_i^2)\biggr)^{1/2}\le Cn^{-1/2}\, .
\]
An application of Lemma \ref{aconc} completes the proof of the lower bound.
\end{proof}

\section{Proof of Theorem \ref{ergmthm}}
Let all notational conventions be the same as in Section \ref{graphsec}. However, instead of a single $H$, consider $l$ graphs $H_1,\ldots, H_l$, and define $T_1,\ldots, T_l$ accordingly. 

Throughout this section, $C$ will denote any constant that may depend only on the graphs $H_1,\ldots, H_l$. Define
\[
f(x) := \beta_1T_1(x)+\cdots + \beta_lT_l(x)\, .
\]
Let $B := 1+|\beta_1|+\cdots +|\beta_l|$, as in the statement of the theorem. 
Let $a$, $b_{(ij)}$ and $c_{(ij)(i'j')}$ be as in Theorem \ref{freethm}. Clearly, 
\[
a \le N^2\sum_{r=1}^l |\beta_r|\le C B N^2\, .
\]
%Let $k_r$ be the number of vertices in $H_r$ and $m_r$ be the number of edges in $H_r$, for $r=1,\ldots, l$.  
By Lemma \ref{test1}, we get the estimates
\[
b_{(ij)} \le CB %\sum_{r=1}^l |\beta_r| m_r\le CB\, ,
\]
and
\[
c_{(ij)(i'j')}\le
\begin{cases}
 CBN^{-1} \ &\text{ if $|\{i,j,i',j'\}| = 2$ or $3$\, ,} \\
 CBN^{-2} & \text{ if $|\{i,j,i',j'\}| = 4$\, .}
\end{cases} 
%\begin{cases}
%N^{-1}\sum_{r=1}^l |\beta_r|m_r(m_r-1) \le CBN^{-1} \ &\text{ if $|\{i,j,i',j'\}| = 2$ or $3$\, ,} \\
%N^{-2}\sum_{r=1}^l |\beta_r| m_r(m_r-1)\le CBN^{-2} & \text{ if $|\{i,j,i',j'\}| = 4$\, .}
%\end{cases} 
\]
Let $\dd_1(\ep),\ldots, \dd_l(\ep)$ be the $\dd(\ep)$'s for $T_1,\ldots, T_l$. Define
\[
\dd(\ep) := \{\beta_1d_1+\cdots +\beta_ld_l \,  : \, d_r\in \dd_i(\ep/\beta_rl), \, r=1,\ldots, l\}\, .
\]
Clearly, for any $x\in [0,1]^n$, there exists $d_1\in \dd_1(\ep/\beta_1 l)$, \ldots, $d_l\in\dd_l(\ep/\beta_l l)$ such that
\begin{align*}
\sum_{i=1}^n (f_i(x)-(\beta_1d_{1i}+\cdots +\beta_ld_{li}))^2 &\le l \sum_{r=1}^l\sum_{i=1}^n \beta_r^2(T_{ri}(x)-d_{ri})^2\le n\ep^2\, .
\end{align*}
Therefore, $\dd(\ep)$ satisfies the requirement of Theorem \ref{freethm}. Also, 
\begin{equation}\label{ergm1}
|\dd(\ep)|\le\prod_{r=1}^l |\dd_r(\ep/\beta_rl)|\, .
\end{equation}
By the bounds on $a$, $b_{(ij)}$ and $c_{(ij)(i'j')}$ obtained above, the following estimates are easy: 
\begin{align*}
&\sum_{(ij)} a c_{(ij)(ij)} \le CB^2 N^3\, , \ \ \sum_{(ij)} b_{(ij)}^2 \le CB^2 N^2\, ,\\
&\sum_{(ij), (i'j')} ac_{(ij)(i'j')}^2 \le CB^3 N^3\, , \\
&\sum_{(ij), (i'j')} b_{(ij)} (b_{(i'j')}+4)c_{(ij)(i'j')}\le CB^3N^2\, ,\\
&\sum_{(ij)} c_{(ij)(ij)}^2 \le CB^2\, , \ \ \sum_{(ij)} c_{(ij)(ij)} \le CBN\, .
\end{align*}
Combining these estimates, we see that the smoothness term is bounded by $C B^2 N^{3/2}$. 
Next, by \eqref{ergm1} and Lemma \ref{test2}, 
\begin{align*}
\log |\dd(\ep)|&\le \sum_{r=1}^l \log |\dd_r(\ep/\beta_rl)|\\
&\le \frac{CB^4 N}{\ep^4} \log \frac{CB^4}{\ep^4}\, .
\end{align*}
Therefore, the complexity term (of Theorem \ref{freethm}) is bounded by
\begin{align*}
CBN^2\ep + \frac{CB^4 N}{\ep^4} \log \frac{CB^4}{\ep^4}\, . 
\end{align*}
Taking 
\[
\ep = \biggl(\frac{B^3\log N}{N}\biggr)^{1/5}\, ,
\]
this gives the bound
\[
CB^{8/5} N^{9/5}(\log N)^{1/5}\biggl(1+\frac{\log B}{\log N}\biggr)\, .
\]
By Theorem \ref{freethm}, this completes the proof of the upper bound. The lower bound follows easily from Theorem \ref{freethm} and the bound on $\sum c_{(ij)(ij)}$ obtained above. This finishes the proof of Theorem \ref{ergmthm}.

\section*{Acknowledgments}
The authors thank Van Vu, Alex Zhai, Yufei Zhao and the anonymous referee for a number of  helpful comments. 
%\vskip.3in

%\vskip.5in

\end{document}